\documentclass[11pt]{amsart}

\usepackage{amsmath,amssymb,amsfonts,amsthm}
\usepackage{mathtools}
\usepackage{mathrsfs}
\usepackage{graphicx}
\usepackage{multirow}
\usepackage[title]{appendix}
\usepackage{xcolor}
\usepackage{textcomp}
\usepackage{booktabs}
\usepackage[hidelinks]{hyperref}

\newtheorem{theorem}{Theorem}[section]
\newtheorem{proposition}[theorem]{Proposition}
\newtheorem{lemma}[theorem]{Lemma}
\newtheorem{corollary}[theorem]{Corollary}
\theoremstyle{definition}
\newtheorem{definition}[theorem]{Definition}
\newtheorem{example}[theorem]{Example}
\theoremstyle{remark}
\newtheorem{remark}[theorem]{Remark}

\title[Metric Completion of Multi-Weighted Conformal Metrics]{Metric Completion and Boundary Geometry of Multi-Weighted Conformal Metrics on Manifolds with Corners}
\author{Muhamad Fahmi bin Zanal Abidin}
\address{Independent Researcher, Kuala Lumpur, Malaysia}
\email{fahmiemail123@gmail.com}
\date{}

\begin{document}
\begin{abstract}
We study the metric completion of conformally singular Riemannian
metrics on the interior of a compact manifold with corners. Let
\(X\) be a compact manifold with corners, let \(g_0\) be a smooth
Riemannian background metric, and let
\(\rho_1,\ldots,\rho_N\) be boundary-defining functions associated with
the boundary hypersurfaces. Given nonnegative weights
\(\alpha_1,\ldots,\alpha_N\), we compare two natural interaction laws,
namely the product metric
\[
g_\Pi
=
\left(
\prod_{i=1}^{N}\rho_i^{-\alpha_i}
\right)g_0
\]
and the sum metric
\[
g_\Sigma
=
\left(
\sum_{i=1}^{N}\rho_i^{-\alpha_i}
\right)g_0.
\]
For a boundary point \(p\), let \(I(p)\) denote its active index set,
and define
\[
A(p)
=
\sum_{i\in I(p)}\alpha_i,
\qquad
M(p)
=
\max_{i\in I(p)}\alpha_i.
\]
We prove that \(p\) occurs at finite distance in the metric completion
of the product metric precisely when
\(A(p)<2\),
whereas for the sum metric the corresponding criterion is
\(M(p)<2\).
Thus the product interaction accumulates singularity across
simultaneously active hypersurfaces, while the sum interaction is
controlled by the largest active weight.
We further prove that, over every accessible boundary point, the active
normal directions collapse to a single point in the metric completion.
This yields a canonical identification of the metric completion with
the subset of the background manifold consisting of the interior
together with its accessible boundary points, and the identification
is shown to be a homeomorphism when the accessible subset carries the
subspace topology inherited from \(X\).
The two interaction laws also produce different boundary-incidence
behaviour. For the product metric, individually accessible boundary
hypersurfaces may determine an inaccessible deeper open face because
the active weights accumulate additively. For the sum metric, an open
face is accessible whenever all of its active hypersurface weights are
subcritical.
Finally, we determine the intrinsic metric geometry induced on
accessible open faces. For the product interaction, an open face
\(F_I^\circ\) with \(A_I=\sum_{i\in I}\alpha_i<2\) carries locally a
snowflake metric with exponent \(1-A_I/2\), and for the sum
interaction, an accessible open face with \(M_I=\max_{i\in
I}\alpha_i<2\) carries locally a snowflake metric with exponent
\(1-M_I/2\). Consequently, if \(F_I^\circ\) has dimension \(m\), its
induced Hausdorff dimension is \(m/(1-A_I/2)\) in the product case and
\(m/(1-M_I/2)\) in the sum case.
These results extend the geometric study of conformal boundary
singularities from a single boundary-defining function to settings in
which several singular hypersurfaces meet and their weights interact
at corners.
\end{abstract}

\maketitle

\noindent
\textit{ORCID:} \href{https://orcid.org/0009-0009-1024-9858}{0009-0009-1024-9858}

\section{Introduction}
\label{sec:introduction}
%================================================

Singular Riemannian metrics arise naturally when a smooth background
metric is modified by a conformal factor that degenerates or blows up
near a distinguished subset. Such metrics occur in geometric analysis,
sub-Riemannian and almost-Riemannian geometry, conformal Grushin
geometry, spectral theory, and models involving singular boundary
behaviour; representative references include
\cite{Cheeger1979,Cheeger1983,Mazzeo1991,Romney2016,CdVDdHT2024}. Even when the underlying manifold is compact, the singular
metric may place portions of its topological boundary either at finite
or infinite metric distance, and the resulting metric completion may
have a geometry substantially different from the original boundary.

A basic model is a conformal singularity associated with a single
boundary-defining function. If \(u\) is a defining function for a
smooth boundary hypersurface and
\[
g=u^{-\alpha}\bar g,
\]
where \(\bar g\) is smooth and nondegenerate up to the boundary, then
the normal distance to the boundary is governed by
\[
\int_0^\varepsilon s^{-\alpha/2}\,ds.
\]
Accordingly, the threshold
\[
\alpha<2
\]
separates the finite-distance and infinite-distance regimes. This
elementary integrability threshold also appears in recent geometric and
spectral analyses of conformally singular boundary metrics; see
\cite{CdVDdHT2024,CdVDT2026}.

The purpose of the present paper is to study what happens when there is
not a single singular boundary hypersurface, but several boundary
hypersurfaces that may meet at corners. At such a point, several
boundary-defining functions vanish simultaneously, and one must specify
how their individual singular contributions interact.

Let \(X\) be a compact manifold with corners, with boundary
hypersurfaces
\[
H_1,\ldots,H_N,
\]
and let
\[
\rho_1,\ldots,\rho_N
\]
be corresponding boundary-defining functions. Let \(g_0\) be a smooth
Riemannian metric on \(X\), and assign to each hypersurface a
nonnegative weight
\[
\alpha_i\geq0.
\]

We compare two natural conformal interaction laws on the interior
\(X^\circ\). The first is the multiplicative, or product, interaction
\begin{equation}
g_\Pi
=
\left(
\prod_{i=1}^{N}\rho_i^{-\alpha_i}
\right)g_0,
\label{eq:introproductmetric}
\end{equation}
and the second is the additive, or sum, interaction
\begin{equation}
g_\Sigma
=
\left(
\sum_{i=1}^{N}\rho_i^{-\alpha_i}
\right)g_0.
\label{eq:introsummetric}
\end{equation}

These two metrics have the same individual singular factors but combine
them by different algebraic rules. The central observation of this
paper is that this algebraic distinction governs the geometry of the
metric completion.

If \(p\in\partial X\), let
\[
I(p)
=
\{i:p\in H_i\}
\]
be the active index set at \(p\). We define
\begin{equation}
A(p)
=
\sum_{i\in I(p)}\alpha_i
\label{eq:introactiveproductweight}
\end{equation}
and
\begin{equation}
M(p)
=
\max_{i\in I(p)}\alpha_i.
\label{eq:introactivesumweight}
\end{equation}

For the product interaction, the active singular orders accumulate, and
the finite-distance criterion is
\[
A(p)<2.
\]
For the sum interaction, the strongest active singularity controls the
finite-distance behaviour, and the corresponding criterion is
\[
M(p)<2.
\]

Thus the familiar one-boundary threshold \(2\) persists, but the
effective exponent at a corner depends on the algebraic interaction
law.

%================================================
\subsection{Relation to Existing Singular Metric Geometry}
\label{subsec:relatedsingulargeometry}
%================================================

The geometry of metrics with singular or degenerate conformal factors
has been studied from several perspectives. Grushin-type and conformal
Grushin geometries provide important examples in which singular
conformal behaviour changes metric scaling and leads naturally to
H\"older control of the induced distance; see \cite{Romney2016}. In
spectral models with conformal boundary blow-up, the singular scaling
also enters the Hausdorff dimension and the Weyl exponent
\cite{CdVDdHT2024}.

There is also an extensive geometric framework for manifolds with
corners and related singular spaces. Joyce developed systematic
foundations for ordinary and generalized manifolds with corners,
including boundary strata and categorical structures
\cite{Joyce2012,Joyce2016}; the later stratified extension is developed
in \cite{Joyce2026}. We use this language to organize
the active boundary hypersurfaces and open faces of \(X\), while the
metric questions considered here concern the additional geometry
created by singular conformal weights.

A broader analytic literature concerns Riemannian spaces with conic,
edge, and stratified singularities. Foundational work of Cheeger
developed the spectral geometry of spaces with cone-like and more
general singular Riemannian structures
\cite{Cheeger1979,Cheeger1983}. Mazzeo developed an elliptic theory for
differential edge operators \cite{Mazzeo1991}, and subsequent work has
studied Hodge theory, harmonic forms, analytic torsion, and related
operators on spaces carrying incomplete or iterated edge metrics
\cite{HunsickerMazzeo2005,MazzeoVertman2011,
AlbinLeichtnamMazzeoPiazza2013}.

These theories provide an important analytic framework for geometric
spaces with singular strata. The present paper addresses a different,
more elementary metric-completion question. The underlying space is a
smooth manifold with corners, while the singularity is introduced
through explicit conformal factors built from several
boundary-defining functions. Our emphasis is on how the algebraic
interaction among these factors determines finite-distance
accessibility, completion topology, and the induced metric geometry of
the surviving open faces.

A closely related recent direction concerns spectral geometry for
conformally singular metrics on manifolds with smooth boundary. Colin
de Verdi\`ere, Dietze, de Hoop, and Tr\'elat study singular metrics
arising from conformal boundary blow-up and derive Weyl asymptotics for
the associated Laplace--Beltrami operators, including applications to
acoustic modes in gas giant planets
\cite{CdVDdHT2024}. Their analysis also connects the singular metric
scaling with Hausdorff dimension.

More recently, Colin de Verdi\`ere, Dietze, and Tr\'elat considered
metrics of the form
\[
g=u^{-\alpha}\bar g,
\]
where \(u\) is a boundary-defining function and the exponent
\(\alpha\) may vary along the boundary
\cite{CdVDT2026}. In the finite-distance regime, the condition
\[
\alpha<2
\]
again plays a fundamental role. Their spectral analysis further shows
that maximal values of a spatially varying boundary exponent can govern
leading Weyl asymptotics in appropriate boundary-dominated regimes.

Related work of Dietze and Read studies concentration of eigenfunctions
for singular Riemannian metrics in supercritical regimes
\cite{DietzeRead2024}, while Dietze investigates the corresponding
critical case
\cite{Dietze2025}. These results demonstrate that conformally singular
boundary metrics form an active area of current spectral and geometric
analysis.

The present problem is related to, but structurally distinct from, this
single-boundary setting. In the recent variable-exponent models, one
studies a single boundary-defining function with an exponent
\[
\alpha=\alpha(y)
\]
that varies with the boundary point \(y\). A quantity such as
\[
\alpha_{\max}
=
\max_{y\in\partial X}\alpha(y)
\]
therefore represents a maximum taken over spatial locations along the
boundary.

By contrast, in the sum interaction studied here,
\[
M_I
=
\max_{i\in I}\alpha_i
\]
is a maximum over several distinct hypersurface weights that are
simultaneously active at the same corner stratum. These are different
mechanisms:
\[
\boxed{
\begin{array}{c}
\text{spatial variation of one boundary exponent}
\\[3pt]
\text{versus}
\\[3pt]
\text{algebraic interaction of several incident boundary weights}.
\end{array}
}
\]

Likewise, the cumulative quantity
\[
A_I
=
\sum_{i\in I}\alpha_i
\]
has no analogue in the ordinary single-boundary model: it arises
specifically because several singular hypersurfaces are simultaneously
active and their conformal factors are multiplied.

The principal geometric question of this paper is therefore not the
existence of the threshold \(2\) itself. Rather, it is how a collection
of individually assigned boundary weights combines at higher-codimension
corners and how the resulting interaction law determines the metric
completion.

%================================================
\subsection{Contribution and Scope}
\label{subsec:introscope}
%================================================

The paper is primarily concerned with metric geometry and metric
completion. We do not develop the spectral theory of the resulting
multi-weight corner metrics here. The recent spectral results for
single-boundary singular metrics
\cite{CdVDdHT2024,DietzeRead2024,Dietze2025,CdVDT2026} suggest that the
interaction laws studied in this paper may lead to corresponding
spectral questions, including Weyl asymptotics and eigenfunction
concentration for singular metrics with several incident boundary
components.

%================================================
\subsection{Main Results}
\label{subsec:intromainresults}
%================================================

Our first result gives a complete pointwise accessibility criterion.

For the product metric,
\[
p\in\partial X
\]
occurs at finite metric distance if and only if
\[
A(p)
=
\sum_{i\in I(p)}\alpha_i
<2.
\]

For the sum metric, the corresponding criterion is
\[
M(p)
=
\max_{i\in I(p)}\alpha_i
<2.
\]

These criteria are proved using explicit finite-length approach curves
in the subcritical regime and metric lower bounds in the critical and
supercritical regimes. The lower-bound arguments allow arbitrary
anisotropic approaches to the corner and therefore do not require the
normal coordinates to vanish at comparable rates.

Our second result concerns the structure of the metric completion.
Whenever a boundary point is accessible, all admissible approaches in
the active normal directions determine the same completion point.
Thus the active normal variables collapse in the completion rather than
creating an additional family of boundary points.

Combining accessibility with normal collapse, we obtain a canonical
identification
\[
\overline{(X^\circ,d_g)}
\cong
X^{\mathrm{acc}},
\]
where
\[
X^{\mathrm{acc}}
=
X^\circ
\cup
\{p\in\partial X:p\text{ is accessible}\}.
\]
Using shrinking-neighbourhood estimates that remain valid across nearby
accessible strata, we show that this canonical bijection is a
homeomorphism when \(X^{\mathrm{acc}}\) carries the subspace topology
inherited from \(X\).

The two interaction laws nevertheless produce different incidence
behaviour. Suppose
\[
F_I^\circ
\]
is an open face with active index set \(I\). For the product metric,
\[
F_I^\circ
\text{ is accessible}
\quad\Longleftrightarrow\quad
\sum_{i\in I}\alpha_i<2.
\]
Consequently, several individually accessible hypersurfaces may meet in
an inaccessible deeper open face.

For the sum metric,
\[
F_I^\circ
\text{ is accessible}
\quad\Longleftrightarrow\quad
\max_{i\in I}\alpha_i<2,
\]
or equivalently,
\[
\alpha_i<2
\qquad
\text{for every }i\in I.
\]
Thus activating additional subcritical hypersurfaces preserves
accessibility. This statement is formulated in terms of open faces and
complete active index sets; it does not assert that every point of an
arbitrary set-theoretic intersection is accessible, since such an
intersection may contain deeper strata lying on additional
hypersurfaces.

Finally, we study the intrinsic metric induced on accessible open
faces. For the product interaction, an accessible face \(F_I^\circ\)
has local boundary metric
\[
d_{\partial,\Pi}(z,z')
\asymp
d_{F_I}(z,z')^{1-A_I/2}.
\]
For the sum interaction,
\[
d_{\partial,\Sigma}(z,z')
\asymp
d_{F_I}(z,z')^{1-M_I/2}
\]
at sufficiently small tangential scales.

The product model admits an exactly homogeneous local model, whereas
the unequal-weight sum model is generally not exactly homogeneous.
Accordingly, the product boundary power law follows from exact dilation
scaling together with nondegeneracy of the model boundary metric, while
the sum law is established by direct two-sided local estimates.

If
\[
m=\dim F_I^\circ,
\]
these snowflake laws yield
\[
\dim_H(F_I^\circ,d_{\partial,\Pi})
=
\frac{m}{1-A_I/2}
\]
in the product case and
\[
\dim_H(F_I^\circ,d_{\partial,\Sigma})
=
\frac{m}{1-M_I/2}
\]
in the sum case.

%================================================
\section{Geometric Setup and Definitions}
\label{sec:setup}
%================================================

In this section we introduce the geometric framework used throughout
the paper. We fix a compact manifold with corners, a smooth background
metric, a collection of boundary-defining functions, and nonnegative
weights attached to the boundary hypersurfaces. We then define the two
singular conformal interaction laws and the notion of accessibility
used in the metric-completion problem.

%================================================
\subsection{Manifolds with Corners and Boundary Hypersurfaces}
\label{subsec:manifoldswithcorners}
%================================================

Let \(X\) be a compact smooth \(n\)-dimensional manifold with corners.
We use the standard local model
\[
[0,\infty)^k\times\mathbb R^{n-k},
\qquad
0\leq k\leq n,
\]
and refer to \cite{Melrose1996,Joyce2012,Joyce2016,Joyce2026} for
background on manifolds with corners and their boundary
stratifications; the boundary-defining-function language used
throughout follows Melrose's original $b$-geometry framework
\cite{Melrose1996}.

We work throughout in the embedded-corners setting: the boundary
hypersurfaces are globally embedded and admit global defining functions
as below. Since \(X\) is compact, only finitely many such hypersurfaces
occur. Let
\[
H_1,\ldots,H_N
\]
denote these boundary hypersurfaces. For each \(i\), choose a
smooth boundary-defining function
\[
\rho_i:X\longrightarrow[0,\infty)
\]
such that
\[
H_i=\rho_i^{-1}(0)
\]
and
\[
d\rho_i\neq0
\qquad
\text{on }H_i.
\]

We assume that the defining functions are chosen so that they are
strictly positive away from their corresponding hypersurfaces:
\[
\rho_i>0
\qquad
\text{on }X\setminus H_i.
\]

Let
\[
X^\circ
=
X\setminus\partial X
\]
denote the interior of \(X\).

We fix a smooth Riemannian metric
\[
g_0
\]
on \(X\). Since \(X\) is compact, \(g_0\) induces a complete metric
space structure on \(X\), which we denote by
\[
d_{g_0}.
\]

The metric \(g_0\) serves only as a smooth background geometry. The
singular metrics considered below are defined on \(X^\circ\), where all
boundary-defining functions are strictly positive.

%================================================

\begin{remark}[Independence of the Choice of Defining Functions]
\label{rem:definingfunctioninvariance}
The qualitative conclusions below do not depend on the particular
choice of boundary-defining functions. If
\(\widetilde\rho_i=a_i\rho_i\), where \(a_i\) is smooth and strictly
positive, then compactness gives constants
\(0<c_i\leq a_i\leq C_i<\infty\). Consequently, the product conformal
factors defined using \(\rho_i\) and \(\widetilde\rho_i\) differ by a
bounded positive multiplicative factor, while the corresponding sum
conformal factors are uniformly comparable term by term. The resulting
metrics are therefore bi-Lipschitz equivalent. Standard metric-space
facts about bi-Lipschitz equivalent metrics and completions then imply
that accessibility and the completion topology are unchanged, while
bi-Lipschitz invariance of Hausdorff dimension preserves the stated
dimension conclusions; see, for example, \cite{BBI2001,heinonen2001}.
\end{remark}

\subsection{Active Index Sets and Open Faces}
\label{subsec:activeindexsets}
%================================================

For each point
\[
p\in X,
\]
define its \emph{active index set} by
\begin{equation}
I(p)
=
\{i\in\{1,\ldots,N\}:p\in H_i\}.
\label{eq:activeindexset}
\end{equation}

Equivalently,
\[
i\in I(p)
\quad\Longleftrightarrow\quad
\rho_i(p)=0.
\]

Thus
\[
I(p)=\varnothing
\]
if and only if
\[
p\in X^\circ.
\]

For an index set
\[
I\subseteq\{1,\ldots,N\},
\]
define the corresponding open face by
\begin{equation}
F_I^\circ
=
\left\{
p\in X:
I(p)=I
\right\}.
\label{eq:openface}
\end{equation}

Equivalently,
\begin{equation}
F_I^\circ
=
\left(
\bigcap_{i\in I}H_i
\right)
\setminus
\left(
\bigcup_{j\notin I}H_j
\right).
\label{eq:openfaceequivalent}
\end{equation}

Whenever
\[
F_I^\circ\neq\varnothing,
\]
it is a smooth manifold without boundary of codimension \(|I|\) in
\(X\).

We emphasize the distinction between the open face
\[
F_I^\circ
\]
and the full set-theoretic intersection
\[
\bigcap_{i\in I}H_i.
\]
The latter may contain deeper strata corresponding to active index sets
\[
J\supsetneq I.
\]

This distinction will be important when discussing accessibility and
boundary incidence.

%================================================
\subsection{Adapted Local Coordinates}
\label{subsec:adaptedcoordinates}
%================================================

Let
\[
p\in F_I^\circ,
\qquad
I=\{i_1,\ldots,i_k\}.
\]

There exists a neighbourhood \(U\) of \(p\) and adapted local
coordinates
\[
(x,z)
=
(x_1,\ldots,x_k,z_1,\ldots,z_{n-k})
\]
such that
\[
x_a\geq0,
\qquad
1\leq a\leq k,
\]
and the active hypersurfaces are represented locally by
\[
H_{i_a}\cap U
=
\{x_a=0\}.
\]

After multiplication by smooth positive functions, the active
boundary-defining functions are locally comparable to the corresponding
normal coordinates. More precisely, after shrinking \(U\) if
necessary, there exist constants
\[
0<c_\rho\leq C_\rho<\infty
\]
such that
\begin{equation}
c_\rho x_a
\leq
\rho_{i_a}(x,z)
\leq
C_\rho x_a,
\qquad
1\leq a\leq k.
\label{eq:definingfunctioncomparison}
\end{equation}

For every inactive index
\[
j\notin I,
\]
we have
\[
\rho_j(p)>0.
\]
Hence, after shrinking \(U\), there exist constants
\[
0<c_j\leq C_j<\infty
\]
such that
\begin{equation}
c_j
\leq
\rho_j
\leq
C_j
\qquad
\text{on }U.
\label{eq:inactivedefiningfunctionbounds}
\end{equation}

The smooth positive-definiteness of \(g_0\) implies uniform local
comparison with the Euclidean coordinate metric. Thus there exist
constants
\[
0<c_0\leq C_0<\infty
\]
such that
\begin{equation}
c_0
\left(
\sum_{a=1}^{k}dx_a^2+|dz|^2
\right)
\leq
g_0
\leq
C_0
\left(
\sum_{a=1}^{k}dx_a^2+|dz|^2
\right)
\label{eq:backgroundmetriccomparison}
\end{equation}
as quadratic forms on \(U\).

In particular, any cross-terms in the coordinate expression of \(g_0\)
are absorbed into the constants \(c_0\) and \(C_0\). No orthogonality
assumption is imposed on the boundary hypersurfaces with respect to
\(g_0\).

%================================================

\begin{remark}[Locality of Comparison Constants]
\label{rem:localcomparisonconstants}
Unless explicitly stated otherwise, constants obtained from adapted
coordinate comparisons are local: they may depend on the fixed
boundary point and on the chosen coordinate neighbourhood. No
uniformity over all boundary points is asserted or needed. This is
distinct from later global estimates obtained from compactness.
\end{remark}

\subsection{Weighted Singular Metrics}
\label{subsec:singularmetrics}
%================================================

Assign to each boundary hypersurface \(H_i\) a nonnegative weight
\[
\alpha_i\geq0.
\]

We study two Riemannian metrics on the interior \(X^\circ\).

The \emph{product interaction metric} is
\begin{equation}
g_\Pi
=
\left(
\prod_{i=1}^{N}
\rho_i^{-\alpha_i}
\right)
g_0.
\label{eq:productmetric}
\end{equation}

The \emph{sum interaction metric} is
\begin{equation}
g_\Sigma
=
\left(
\sum_{i=1}^{N}
\rho_i^{-\alpha_i}
\right)
g_0.
\label{eq:summetric}
\end{equation}

Both metrics are smooth and positive definite on \(X^\circ\), but may
be singular as one approaches \(\partial X\).

For an active index set \(I\), define the two effective weights
\begin{equation}
A_I
=
\sum_{i\in I}\alpha_i
\label{eq:effectiveproductweight}
\end{equation}
and
\begin{equation}
M_I
=
\max_{i\in I}\alpha_i.
\label{eq:effectivesumweight}
\end{equation}

At a point \(p\in\partial X\), we write
\begin{equation}
A(p)
=
A_{I(p)}
=
\sum_{i\in I(p)}\alpha_i
\label{eq:pointwiseproductweight}
\end{equation}
and
\begin{equation}
M(p)
=
M_{I(p)}
=
\max_{i\in I(p)}\alpha_i.
\label{eq:pointwisesumweight}
\end{equation}

For completeness, we set
\[
A_\varnothing=0.
\]

The quantity \(M_I\) will only be used for nonempty active sets.

%================================================
\subsection{Local Model Comparisons}
\label{subsec:localmodelcomparisons}
%================================================

Let
\[
p\in F_I^\circ
\]
and choose adapted coordinates as above.

For the product interaction, the inactive defining functions are
uniformly bounded above and below by positive constants. Combining
\eqref{eq:definingfunctioncomparison},
\eqref{eq:inactivedefiningfunctionbounds}, and
\eqref{eq:backgroundmetriccomparison}, there exist constants
\[
0<c_\Pi\leq C_\Pi<\infty
\]
such that
\begin{equation}
c_\Pi
\left(
\prod_{i\in I}
x_i^{-\alpha_i}
\right)
\left(
\sum_{i\in I}dx_i^2+|dz|^2
\right)
\leq
g_\Pi
\leq
C_\Pi
\left(
\prod_{i\in I}
x_i^{-\alpha_i}
\right)
\left(
\sum_{i\in I}dx_i^2+|dz|^2
\right).
\label{eq:localproductcomparison}
\end{equation}

After shrinking the adapted coordinate neighborhood if necessary, we
may take \(U\) to be a product coordinate cylinder of the form
\[
U=\{r(x)<r_1\}\times\{|z-z_0|<r_1\}
\]
for some \(r_1>0\). For every \(0<r_*<r_1\), the smaller normal region
\(\{r(x)\leq r_*\}\) can therefore be left in the normal variables only
by crossing the level \(r(x)=r_*\); if a curve instead exits through
the fixed tangential boundary of the coordinate cylinder, it incurs a
uniform positive \(g_0\)-distance and hence, by the local lower metric
comparison, a uniform positive \(g\)-length cost. This is the
stay-versus-exit alternative used in the necessity argument below.

For the sum interaction, the inactive terms remain uniformly bounded.
Thus, after shrinking the coordinate neighbourhood if necessary, there
exist constants
\[
0<c_\Sigma\leq C_\Sigma<\infty
\]
such that
\begin{equation}
c_\Sigma
\left(
1+
\sum_{i\in I}
x_i^{-\alpha_i}
\right)
\left(
\sum_{i\in I}dx_i^2+|dz|^2
\right)
\leq
g_\Sigma
\leq
C_\Sigma
\left(
1+
\sum_{i\in I}
x_i^{-\alpha_i}
\right)
\left(
\sum_{i\in I}dx_i^2+|dz|^2
\right).
\label{eq:localsumcomparison}
\end{equation}

Near a genuine singular boundary point for which at least one active
weight is positive, the additive constant \(1\) may be absorbed into
the singular sum after restricting to a sufficiently small
neighbourhood.

More explicitly, the inactive boundary-defining functions are uniformly
bounded above and below by positive constants on the chosen coordinate
neighborhood. Hence their contribution to the sum conformal factor may
be written as a function \(B(x,z)\) satisfying
\[
0<c_{\mathrm{inactive}}
\leq B(x,z)
\leq C_{\mathrm{inactive}}<\infty.
\]
Consequently,
\[
\min\{1,c_{\mathrm{inactive}}\}
\left(
1+\sum_{i\in I}x_i^{-\alpha_i}
\right)
\leq
B(x,z)+\sum_{i\in I}x_i^{-\alpha_i}
\leq
\max\{1,C_{\mathrm{inactive}}\}
\left(
1+\sum_{i\in I}x_i^{-\alpha_i}
\right).
\]
Thus the inactive terms and the nonsingular additive constant are
absorbed into uniform comparison constants.
 We retain it in
\eqref{eq:localsumcomparison}
because this formulation also covers active hypersurfaces with zero
weight.

These local comparison estimates are uniform in direction. In
particular, they remain valid for anisotropic approaches in which the
active normal coordinates tend to zero at different rates.

%================================================
\subsection{Metric Completion and Accessibility}
\label{subsec:accessibilitydefinition}
%================================================

Let
\[
g\in\{g_\Pi,g_\Sigma\},
\]
and let
\[
d_g
\]
denote the corresponding Riemannian distance on \(X^\circ\).

We write
\[
\overline{(X^\circ,d_g)}
\]
for the metric completion of \((X^\circ,d_g)\).

\begin{definition}[Accessible Boundary Point]
\label{def:accessiblepoint}

A point
\[
p\in\partial X
\]
is called \emph{accessible with respect to \(g\)} if there exists a
\(d_g\)-Cauchy sequence
\[
\{p_n\}\subset X^\circ
\]
such that
\[
p_n\longrightarrow p
\]
in the background topology of \(X\).

\end{definition}

\begin{remark}
The finite-length formulation introduced below is equivalent to
Definition~\ref{def:accessiblepoint} for the metrics considered in this
paper. The equivalence is proved in
Section~\ref{subsec:accessibilityformulationequivalence}; thus the two
uses of the word ``accessible'' do not define different classes of
boundary points.
\end{remark}

Define the accessible boundary by
\begin{equation}
\partial_g^{\mathrm{acc}}X
=
\left\{
p\in\partial X:
p\text{ is accessible with respect to }g
\right\},
\label{eq:accessibleboundary}
\end{equation}
and define the accessible subset
\begin{equation}
X_g^{\mathrm{acc}}
=
X^\circ
\cup
\partial_g^{\mathrm{acc}}X.
\label{eq:accessiblesubset}
\end{equation}

When the metric under consideration is clear, we write simply
\[
X^{\mathrm{acc}}.
\]

\begin{remark}[Accessibility and Finite-Length Approachability]
\label{rem:accessibilityequivalence}

The metric-completion formulation in
Definition~\ref{def:accessiblepoint}
is the primary definition of accessibility used throughout this paper.

We will also use a related notion. A point
\[
p\in\partial X
\]
is called \emph{finite-length approachable} if there exists an
absolutely continuous curve
\[
\gamma:[0,1)\longrightarrow X^\circ
\]
such that
\[
\gamma(t)\longrightarrow p
\qquad
\text{as }t\uparrow1
\]
in the background topology and
\[
L_g(\gamma)<\infty.
\]

Every finite-length approachable point is accessible. Indeed, choose
\[
t_n\uparrow1
\]
and set
\[
p_n=\gamma(t_n).
\]
For \(m>n\),
\[
d_g(p_n,p_m)
\leq
L_g\bigl(\gamma|_{[t_n,t_m]}\bigr).
\]
Since \(\gamma\) has finite total length, the lengths of its terminal
subarcs tend to zero. Hence
\[
\{p_n\}
\]
is a \(d_g\)-Cauchy sequence converging to \(p\) in the background
topology.

For the product and sum metrics studied here, the converse will follow
from the accessibility criteria proved in
Section~\ref{sec:accessibility}. Consequently, the two notions will
ultimately coincide. Until this equivalence has been established,
however, the term \emph{accessible} always refers to
Definition~\ref{def:accessiblepoint}.

\end{remark}

%================================================
\subsection{Accessible Open Faces}
\label{subsec:accessibleopenfaces}
%================================================

\begin{definition}[Accessible Open Face]
\label{def:accessibleopenface}

An open face
\[
F_I^\circ
\]
is called \emph{accessible with respect to \(g\)} if every point of
\(F_I^\circ\) is accessible with respect to \(g\).

\end{definition}

Because the active index set is constant on \(F_I^\circ\), the
accessibility criteria established later will imply that accessibility
is constant along each open face for the fixed-weight metrics considered
in this paper.

More precisely, we will prove
\[
F_I^\circ
\text{ is accessible for }g_\Pi
\quad\Longleftrightarrow\quad
A_I<2,
\]
and
\[
F_I^\circ
\text{ is accessible for }g_\Sigma
\quad\Longleftrightarrow\quad
M_I<2.
\]

These statements concern the open face \(F_I^\circ\), whose points have
complete active index set exactly \(I\). They should not be confused
with statements about the entire set-theoretic intersection
\[
\bigcap_{i\in I}H_i,
\]
which may contain deeper open faces with additional active
hypersurfaces.

%================================================
\subsection{The Completion Map}
\label{subsec:completionmapsetup}
%================================================

The compactness of \(X\) and the global lower metric bounds established
later will imply that every \(d_g\)-Cauchy sequence is also Cauchy with
respect to the smooth background metric \(g_0\). Consequently, every
point of the metric completion has a unique background limit in \(X\).

This will define a canonical map
\begin{equation}
\Phi_g:
\overline{(X^\circ,d_g)}
\longrightarrow
X_g^{\mathrm{acc}},
\label{eq:canonicalcompletionmap}
\end{equation}
which sends a completion class represented by a \(d_g\)-Cauchy
sequence to its unique background limit.

The construction, bijectivity, and topological properties of this map
will be established in Section~\ref{sec:globalcompletion}. In
particular, the normal-collapse results of
Section~\ref{sec:normalcollapse}
will show that different interior approaches to the same accessible
boundary point determine the same completion point.

%================================================
\section{Accessibility Criteria}
\label{sec:accessibility}
%================================================

We now determine precisely which boundary points occur at finite
distance for the product and sum interaction metrics.

Let
\[
p\in F_I^\circ,
\]
where
\[
I=I(p)
\]
is the active index set. Recall the effective weights
\[
A_I=\sum_{i\in I}\alpha_i,
\qquad
M_I=\max_{i\in I}\alpha_i.
\]

The main results are
\[
p\text{ is accessible for }g_\Pi
\quad\Longleftrightarrow\quad
A_I<2,
\]
and
\[
p\text{ is accessible for }g_\Sigma
\quad\Longleftrightarrow\quad
M_I<2.
\]

The sufficiency direction is obtained by constructing an explicit
finite-length diagonal approach. The necessity direction is proved by
a lower bound valid for arbitrary absolutely continuous curves and,
subsequently, for arbitrary Cauchy sequences approaching the boundary.

The estimates do not require the active normal coordinates to approach
zero at comparable rates. In particular, no conical-wedge restriction
is imposed.

%================================================
\subsection{Local Radial Estimates}
\label{subsec:localradialestimates}
%================================================

Fix
\[
p\in F_I^\circ
\]
and choose adapted coordinates
\[
(x,z)
=
(x_1,\ldots,x_k,z)
\]
on a neighbourhood \(U\) of \(p\), where
\[
k=|I|
\]
and
\[
p=(0,\ldots,0,z_0).
\]

Define the active normal radius
\begin{equation}
r(x)
=
\left(
\sum_{i=1}^{k}x_i^2
\right)^{1/2}.
\label{eq:activenormalradius}
\end{equation}

For every \(i\),
\[
0\leq x_i\leq r.
\]

Hence, since all weights are nonnegative,
\begin{equation}
\prod_{i=1}^{k}x_i^{-\alpha_i}
\geq
r^{-A_I}.
\label{eq:productradialweightlower}
\end{equation}

For the sum interaction, choose an index
\[
j\in I
\]
such that
\[
\alpha_j=M_I.
\]
Since
\[
x_j\leq r,
\]
we have
\begin{equation}
\sum_{i=1}^{k}x_i^{-\alpha_i}
\geq
x_j^{-M_I}
\geq
r^{-M_I}.
\label{eq:sumradialweightlower}
\end{equation}

Combining these inequalities with the local metric comparisons from
Section~\ref{sec:setup}, after possibly shrinking \(U\), there exists
a constant \(c>0\) such that
\begin{equation}
g_\Pi
\geq
c\,r^{-A_I}
\left(
|dx|^2+|dz|^2
\right)
\label{eq:productradialmetriclower}
\end{equation}
and
\begin{equation}
g_\Sigma
\geq
c\,r^{-M_I}
\left(
|dx|^2+|dz|^2
\right).
\label{eq:sumradialmetriclower}
\end{equation}

These inequalities hold throughout the adapted neighbourhood and are
uniform with respect to the direction of approach.

Let
\[
\gamma(t)
=
(x(t),z(t))
\]
be an absolutely continuous curve in \(U\cap X^\circ\). Since the
Euclidean norm is Lipschitz,
\[
r(t)=|x(t)|
\]
is absolutely continuous and satisfies
\begin{equation}
|\dot r(t)|
\leq
|\dot x(t)|
\label{eq:radialvelocitybound}
\end{equation}
for almost every \(t\).

Therefore,
\begin{equation}
|\dot\gamma(t)|_{g_\Pi}
\geq
c^{1/2}
r(t)^{-A_I/2}
|\dot r(t)|
\label{eq:productradialspeedlower}
\end{equation}
and
\begin{equation}
|\dot\gamma(t)|_{g_\Sigma}
\geq
c^{1/2}
r(t)^{-M_I/2}
|\dot r(t)|.
\label{eq:sumradialspeedlower}
\end{equation}

%================================================
\subsection{A Weighted Radial Variation Lemma}
\label{subsec:weightedradialvariation}
%================================================

The lower-bound arguments for both interaction laws use the same
one-dimensional estimate. We record it once to avoid repeating a
singular-endpoint chain-rule argument.

\begin{lemma}[Weighted Radial Variation]
\label{lem:weightedradialvariation}
Let \(r:[a,b]\to(0,\infty)\) be absolutely continuous and let
\(\beta\geq0\). Define
\[
H_\beta(s)=
\begin{cases}
\dfrac{s^{1-\beta}}{1-\beta}, & \beta\neq1,\\[6pt]
\log s, & \beta=1.
\end{cases}
\]
Then
\begin{equation}
\int_a^b r(t)^{-\beta}|\dot r(t)|\,dt
\geq
\left|H_\beta(r(b))-H_\beta(r(a))\right|.
\label{eq:weightedradialvariation}
\end{equation}
No monotonicity of \(r\) is required.
\end{lemma}

\begin{proof}
Since \(r\) is continuous and strictly positive on the compact interval
\([a,b]\), there exists \(m>0\) such that \(r(t)\geq m\) for all
\(t\in[a,b]\). Hence \(H_\beta\) is \(C^1\) with bounded derivative on
the compact image \(r([a,b])\subset(0,\infty)\). Therefore
\(H_\beta\circ r\) is absolutely continuous and the chain rule gives
\[
H_\beta(r(b))-H_\beta(r(a))
=
\int_a^b r(t)^{-\beta}\dot r(t)\,dt.
\]
Taking absolute values and applying the triangle inequality yields
\[
\left|H_\beta(r(b))-H_\beta(r(a))\right|
\leq
\int_a^b r(t)^{-\beta}|\dot r(t)|\,dt,
\]
which proves the claim.
\end{proof}

\subsection{Product Interaction}
\label{subsec:productaccessibility}
%================================================

\begin{theorem}[Accessibility Criterion for the Product Interaction]
\label{thm:productaccessibility}

Let
\[
p\in F_I^\circ.
\]
Then the following are equivalent:

\begin{enumerate}
    \item \(p\) is accessible for \(g_\Pi\);

    \item \(p\) is finite-length approachable for \(g_\Pi\);

    \item
    \[
    A_I<2.
    \]
\end{enumerate}

\end{theorem}

\begin{proof}

We first prove
\[
(3)\Longrightarrow(2).
\]

Assume
\[
A_I<2.
\]

Choose adapted coordinates centred at
\[
p=(0,\ldots,0,z_0).
\]
For sufficiently small \(\varepsilon>0\), consider the diagonal curve
\[
\gamma:(0,\varepsilon]\longrightarrow X^\circ,
\qquad
\gamma(s)
=
(s,\ldots,s,z_0).
\]

The curve converges to \(p\) in the background topology as
\[
s\downarrow0.
\]

Along this curve,
\[
x_i=s
\qquad
\text{for every }i\in I,
\]
and
\[
|\dot x|^2=k.
\]

By the upper local comparison for \(g_\Pi\),
\[
|\dot\gamma(s)|_{g_\Pi}
\leq
C s^{-A_I/2}
\]
for some constant \(C>0\). Hence
\[
L_{g_\Pi}(\gamma)
\leq
C
\int_0^\varepsilon
s^{-A_I/2}\,ds.
\]

Since
\[
A_I<2,
\]
the integral converges. Therefore \(p\) is finite-length approachable.

By Remark~\ref{rem:accessibilityequivalence}, every finite-length
approachable point is accessible. Thus
\[
(2)\Longrightarrow(1).
\]

It remains to prove
\[
(1)\Longrightarrow(3).
\]

Suppose, to the contrary, that
\[
A_I\geq2
\]
and that \(p\) is accessible.

Then there exists a \(d_{g_\Pi}\)-Cauchy sequence
\[
q_n\in X^\circ
\]
such that
\[
q_n\longrightarrow p
\]
in the background topology.

For sufficiently large \(n\), all \(q_n\) lie in the adapted
neighbourhood \(U\). Write
\[
q_n=(x^{(n)},z^{(n)})
\]
and set
\[
r_n
=
|x^{(n)}|.
\]
Then
\[
r_n\longrightarrow0.
\]

Choose a fixed radius
\[
r_*>0
\]
small enough that
\[
\{(x,z)\in U:r(x)\leq r_*\}
\]
lies inside the adapted coordinate neighbourhood.

Since
\[
r_n\to0,
\]
we may choose recursively a subsequence, relabelled again by
\(\{r_n\}\), such that
\[
r_{n+1}<r_n
\]
for every \(n\).

Consider any absolutely continuous curve \(\gamma\) joining \(q_n\) to
\(q_m\), where \(m>n\). If the curve remains in \(U\), the radial speed
estimate gives
\[
L_{g_\Pi}(\gamma)
\geq
c^{1/2}
\int
r^{-A_I/2}|\dot r|\,dt.
\]

By Lemma~\ref{lem:weightedradialvariation}, applied to the radial
coordinate on the curve (no monotonicity is required),
\[
\int
r^{-A_I/2}|\dot r|\,dt
\geq
\left|
\int_{r_m}^{r_n}
s^{-A_I/2}\,ds
\right|.
\]
Therefore
\begin{equation}
L_{g_\Pi}(\gamma)
\geq
c^{1/2}
\int_{r_m}^{r_n}
s^{-A_I/2}\,ds.
\label{eq:productannularlowerbound}
\end{equation}

If instead the curve leaves the coordinate neighbourhood through the
region
\[
r\geq r_*,
\]
then, on the subcurve up to the first crossing of \(r=r_*\),
Lemma~\ref{lem:weightedradialvariation} gives
\begin{equation}
L_{g_\Pi}(\gamma)
\geq
c^{1/2}
\int_{r_n}^{r_*}
s^{-A_I/2}\,ds.
\label{eq:productescape lower}
\end{equation}

We now use the assumption
\[
A_I\geq2.
\]

Define
\[
\Phi_\Pi(r)
=
\int_r^{r_*}
s^{-A_I/2}\,ds.
\]
Then
\[
\Phi_\Pi(r)\longrightarrow+\infty
\qquad
\text{as }r\downarrow0.
\]

Since \(\Phi_\Pi\) is continuous, strictly decreasing as a function
of \(r\), and unbounded as \(r\downarrow0\), the indices may be
chosen greedily: after choosing one term, choose the next far enough
along the sequence that the primitive has increased by at least one.
After relabelling, we may therefore assume
\[
\Phi_\Pi(r_{n+1})
-
\Phi_\Pi(r_n)
\geq1
\]
for every \(n\).

For any curve joining \(q_n\) to \(q_{n+1}\), there are three
possibilities. If the curve remains inside the adapted neighbourhood,
then \eqref{eq:productannularlowerbound} gives
\[
L_{g_\Pi}(\gamma)\geq c^{1/2}.
\]
If it exits through the normal boundary \(r=r_*\), then
\eqref{eq:productescape lower} gives an even larger lower bound for all
sufficiently large \(n\). Finally, if it exits through the fixed
tangential boundary of the product coordinate cylinder, the
stay-versus-exit estimate from
Section~\ref{subsec:localmodelcomparisons} gives a uniform lower bound
\(E_{\mathrm{tan}}>0\), independent of \(n\). Therefore
\[
d_{g_\Pi}(q_n,q_{n+1})
\geq
\min\{c^{1/2},E_{\mathrm{tan}}\}>0
\]
along this subsequence, contradicting the assumption that
\[
\{q_n\}
\]
is \(d_{g_\Pi}\)-Cauchy.

Therefore
\[
A_I<2.
\]

This proves the equivalence of the three conditions.

\end{proof}

\begin{remark}[Anisotropic Approaches]
\label{rem:productanisotropicaccessibility}

The necessity argument does not assume that the active normal
coordinates vanish at comparable rates. The key estimate
\[
\prod_{i\in I}x_i^{-\alpha_i}
\geq
r^{-A_I}
\]
follows solely from
\[
x_i\leq r
\]
and
\[
\alpha_i\geq0.
\]
Thus paths for which, for example,
\[
x_1\ll x_2
\]
are included automatically. No restriction of the form
\[
c\leq\frac{x_i}{x_j}\leq C
\]
is required.

\end{remark}

%================================================
\subsection{Sum Interaction}
\label{subsec:sumaccessibility}
%================================================

\begin{theorem}[Accessibility Criterion for the Sum Interaction]
\label{thm:sumaccessibility}

Let
\[
p\in F_I^\circ.
\]
Then the following are equivalent:

\begin{enumerate}
    \item \(p\) is accessible for \(g_\Sigma\);

    \item \(p\) is finite-length approachable for \(g_\Sigma\);

    \item
    \[
    M_I<2.
    \]
\end{enumerate}

\end{theorem}

\begin{proof}

We first prove
\[
(3)\Longrightarrow(2).
\]

Assume
\[
M_I<2.
\]

Choose the diagonal curve
\[
\gamma(s)
=
(s,\ldots,s,z_0),
\qquad
0<s\leq\varepsilon.
\]

Along this curve,
\[
\sum_{i\in I}s^{-\alpha_i}
\leq
k\,s^{-M_I}
\]
for
\[
0<s\leq1.
\]
The inactive terms remain uniformly bounded and may be absorbed into
the comparison constant.

Therefore
\[
|\dot\gamma(s)|_{g_\Sigma}
\leq
C s^{-M_I/2}
\]
for some constant \(C>0\), and hence
\[
L_{g_\Sigma}(\gamma)
\leq
C
\int_0^\varepsilon
s^{-M_I/2}\,ds.
\]

Since
\[
M_I<2,
\]
the integral converges. Thus \(p\) is finite-length approachable and
therefore accessible.

It remains to prove necessity.

Suppose
\[
M_I\geq2
\]
and assume, for contradiction, that \(p\) is accessible.

Let
\[
q_n=(x^{(n)},z^{(n)})
\]
be a \(d_{g_\Sigma}\)-Cauchy sequence converging to \(p\) in the
background topology, and set
\[
r_n=|x^{(n)}|.
\]
Then
\[
r_n\to0.
\]

From \eqref{eq:sumradialspeedlower}, every absolutely continuous curve
remaining in the adapted neighbourhood satisfies
\[
L_{g_\Sigma}(\gamma)
\geq
c^{1/2}
\int
r^{-M_I/2}|\dot r|\,dt.
\]

Define
\[
\Phi_\Sigma(r)
=
\int_r^{r_*}
s^{-M_I/2}\,ds.
\]
Because
\[
M_I\geq2,
\]
we have
\[
\Phi_\Sigma(r)\longrightarrow+\infty
\qquad
\text{as }r\downarrow0.
\]

Since \(\Phi_\Sigma\) is continuous, strictly decreasing as a
function of \(r\), and unbounded as \(r\downarrow0\), choose the
subsequence greedily so that
\[
\Phi_\Sigma(r_{n+1})-\Phi_\Sigma(r_n)\geq1.
\]
For a curve remaining in the adapted neighbourhood,
Lemma~\ref{lem:weightedradialvariation} gives
\[
L_{g_\Sigma}(\gamma)\geq c^{1/2}.
\]
If the curve exits through the normal boundary \(r=r_*\), applying the
same lemma to the subcurve up to its first crossing of \(r=r_*\)
gives a uniformly positive, and eventually larger, lower bound. If it
instead exits through the fixed tangential boundary of the product
coordinate cylinder, the stay-versus-exit estimate from
Section~\ref{subsec:localmodelcomparisons} gives a uniform lower bound
\(E_{\mathrm{tan}}>0\). Thus
\[
d_{g_\Sigma}(q_n,q_{n+1})
\geq
\min\{c^{1/2},E_{\mathrm{tan}}\}>0,
\]
contradicting the Cauchy property.

Therefore
\[
M_I<2.
\]

\end{proof}

\begin{remark}[Dominant Active Weight]
\label{rem:dominantactiveweight}

The sum criterion depends only on
\[
M_I
=
\max_{i\in I}\alpha_i.
\]
Indeed, if
\[
j\in I
\]
satisfies
\[
\alpha_j=M_I,
\]
then
\[
\sum_{i\in I}x_i^{-\alpha_i}
\geq
x_j^{-M_I}
\geq
r^{-M_I}.
\]
Thus a single active hypersurface with critical or supercritical weight
is sufficient to place the entire open face at infinite metric
distance.

Conversely, if every active weight is subcritical, then
\[
M_I<2,
\]
and the diagonal finite-length approach proves accessibility.

\end{remark}

%================================================
\subsection{Pointwise and Facewise Criteria}
\label{subsec:pointwisefacewiseaccessibility}
%================================================

The preceding theorems immediately give the pointwise classification
\begin{equation}
p\in\partial_{g_\Pi}^{\mathrm{acc}}X
\quad\Longleftrightarrow\quad
A(p)<2
\label{eq:productpointwiseaccessibility}
\end{equation}
and
\begin{equation}
p\in\partial_{g_\Sigma}^{\mathrm{acc}}X
\quad\Longleftrightarrow\quad
M(p)<2.
\label{eq:sumpointwiseaccessibility}
\end{equation}

Since the active index set is constant on an open face
\[
F_I^\circ,
\]
we also obtain
\begin{equation}
F_I^\circ
\text{ is accessible for }g_\Pi
\quad\Longleftrightarrow\quad
A_I<2
\label{eq:productfaceaccessibility}
\end{equation}
and
\begin{equation}
F_I^\circ
\text{ is accessible for }g_\Sigma
\quad\Longleftrightarrow\quad
M_I<2.
\label{eq:sumfaceaccessibility}
\end{equation}

For the sum interaction,
\[
M_I<2
\]
is equivalent to
\[
\alpha_i<2
\qquad
\text{for every }i\in I.
\]

For the product interaction, however,
\[
\alpha_i<2
\qquad
\text{for every }i\in I
\]
does not imply
\[
A_I<2.
\]
This is the source of the different incidence behaviour developed later
in the paper.

%================================================
\subsection{Equivalence of the Two Accessibility Formulations}
\label{subsec:accessibilityformulationequivalence}
%================================================

Theorems~\ref{thm:productaccessibility} and
\ref{thm:sumaccessibility}
also complete the comparison between the two notions introduced in
Section~\ref{sec:setup}.

\begin{corollary}[Accessibility Equals Finite-Length Approachability]
\label{cor:accessibilityequivalence}

For either
\[
g=g_\Pi
\]
or
\[
g=g_\Sigma,
\]
a boundary point
\[
p\in\partial X
\]
is accessible in the sense of
Definition~\ref{def:accessiblepoint}
if and only if it is finite-length approachable.

\end{corollary}

\begin{proof}

Finite-length approachability implies accessibility by
Remark~\ref{rem:accessibilityequivalence}.

Conversely, if \(p\) is accessible, then
Theorem~\ref{thm:productaccessibility} or
Theorem~\ref{thm:sumaccessibility}, according to the metric under
consideration, implies that the corresponding effective weight is
strictly less than \(2\). The diagonal construction in the sufficiency
part of the same theorem then provides an explicit finite-length curve
approaching \(p\).

\end{proof}

%================================================
\subsection{Critical and Supercritical Regimes}
The threshold \(2\) is sharp. The radial lower bounds give logarithmic
divergence at effective weight \(2\) and polynomial divergence above it.
Thus \(F_I^\circ\) is accessible exactly when \(A_I<2\) for \(g_\Pi\) and
\(M_I<2\) for \(g_\Sigma\).

\section{Collapse of Active Normal Directions}
\label{sec:normalcollapse}
%================================================

Having determined which boundary points occur at finite metric distance,
we now study the local geometry of approaches to an accessible open
face.

The main result of this section is that, over a fixed accessible
background point, the active normal directions collapse to a single
point in the metric completion. Thus different rates of approach in the
normal variables do not produce distinct completion points.

The argument uses adapted corner coordinates and the local model
comparisons established in Section~\ref{sec:setup}. The estimates are
uniform under strongly anisotropic approaches and do not require the
active normal coordinates to vanish at comparable rates.

Throughout this section, fix
\[
p\in F_I^\circ
\]
and choose adapted local coordinates
\[
(x,z)
=
(x_1,\ldots,x_k,z)
\]
on a neighbourhood \(U\) of \(p\), where
\[
k=|I|
\]
and
\[
p=(0,\ldots,0,z_0).
\]

For the product interaction, write
\[
A_I
=
\sum_{i\in I}\alpha_i,
\]
and for the sum interaction write
\[
M_I
=
\max_{i\in I}\alpha_i.
\]

%================================================
\subsection{Shrinking Normal Boxes}
\label{subsec:shrinkingnormalboxes}
%================================================

Fix the tangential point
\[
z_0.
\]

For \(t>0\), define the normal box
\begin{equation}
Q_t(z_0)
=
\left\{
(x,z_0):
0<x_i\leq t
\text{ for every }i\in I
\right\}.
\label{eq:normalbox}
\end{equation}

The set \(Q_t(z_0)\) consists of interior points approaching the same
tangential location through arbitrary configurations of the active
normal coordinates.

We first show that its diameter tends to zero whenever the
corresponding face is accessible.

%================================================
\subsection{Product Interaction}
\label{subsec:productnormalcollapse}
%================================================

\begin{lemma}[Shrinking Normal Diameter for the Product Metric]
\label{lem:productshrinkingnormaldiameter}

Assume
\[
A_I<2.
\]

Then there exist constants
\[
C>0,
\qquad
t_0>0,
\]
such that
\begin{equation}
\operatorname{diam}_{g_\Pi}
Q_t(z_0)
\leq
Ct^{1-A_I/2}
\label{eq:productnormaldiameter}
\end{equation}
for every
\[
0<t<t_0.
\]

In particular,
\[
\operatorname{diam}_{g_\Pi}
Q_t(z_0)
\longrightarrow0
\qquad
\text{as }t\downarrow0.
\]

\end{lemma}

\begin{proof}

Fix
\[
q=(x,z_0)\in Q_t(z_0),
\]
and define the diagonal reference point
\[
q_t
=
(t,\ldots,t,z_0).
\]

Consider the straight interpolation
\begin{equation}
\gamma(s)
=
\bigl(
x+s(t\mathbf 1-x),
z_0
\bigr),
\qquad
0\leq s\leq1,
\label{eq:productnormalinterpolation}
\end{equation}
where
\[
\mathbf 1=(1,\ldots,1).
\]

For each active coordinate,
\[
\gamma_i(s)
=
(1-s)x_i+st.
\]

Since
\[
x_i>0,
\]
we have
\begin{equation}
\gamma_i(s)
\geq
st.
\label{eq:interpolationlowerbound}
\end{equation}

Thus
\[
\prod_{i\in I}
\gamma_i(s)^{-\alpha_i}
\leq
(st)^{-A_I}.
\]

Moreover,
\[
|\dot\gamma(s)|
=
|t\mathbf1-x|
\leq
Ct,
\]
where \(C\) depends only on \(k\).

Using the local upper comparison for the product metric,
\[
|\dot\gamma(s)|_{g_\Pi}
\leq
Ct
(st)^{-A_I/2}.
\]

Therefore
\begin{align}
L_{g_\Pi}(\gamma)
&\leq
Ct
\int_0^1
(st)^{-A_I/2}\,ds
\\
&=
Ct^{1-A_I/2}
\int_0^1
s^{-A_I/2}\,ds.
\end{align}

Since
\[
A_I<2,
\]
the integral is finite. Hence
\begin{equation}
d_{g_\Pi}(q,q_t)
\leq
Ct^{1-A_I/2}.
\label{eq:productnormalanchorestimate}
\end{equation}

Now let
\[
q,q'\in Q_t(z_0).
\]

By the triangle inequality,
\[
d_{g_\Pi}(q,q')
\leq
d_{g_\Pi}(q,q_t)
+
d_{g_\Pi}(q_t,q'),
\]
and hence
\[
d_{g_\Pi}(q,q')
\leq
Ct^{1-A_I/2}.
\]

Taking the supremum over all
\[
q,q'\in Q_t(z_0)
\]
proves the result.

\end{proof}

\begin{remark}[Uniformity under Anisotropic Normal Data]
\label{rem:anisotropicnormalcollapse}

The estimate above is uniform even when the initial normal coordinates
are highly anisotropic.

For example, one may have
\[
x_1=t,
\qquad
x_2=t^{10},
\]
or more generally coordinates tending to zero at arbitrarily different
rates.

The key inequality
\[
\gamma_i(s)\geq st
\]
holds for every active coordinate and every starting point in
\(Q_t(z_0)\). Thus the conformal factor is bounded above by the
integrable singularity
\[
(st)^{-A_I},
\]
independently of the relative sizes of the initial coordinates.

In particular, the interpolation path does not move toward a more
singular region. Every normal coordinate is monotone nondecreasing
toward the common depth \(t\).

\end{remark}

\begin{theorem}[Normal Collapse for the Product Interaction]
\label{thm:productnormalcollapse}

Let
\[
p=(0,z_0)\in F_I^\circ
\]
and assume
\[
A_I<2.
\]

Suppose
\[
q_n=(x^{(n)},z^{(n)})
\]
and
\[
q_n'=(y^{(n)},w^{(n)})
\]
are two sequences in \(X^\circ\) satisfying
\[
q_n\longrightarrow p,
\qquad
q_n'\longrightarrow p
\]
in the background topology.

Then
\begin{equation}
d_{g_\Pi}(q_n,q_n')
\longrightarrow0.
\label{eq:productnormalcollapsesequences}
\end{equation}

Consequently, all interior sequences converging to the same accessible
background point determine the same point of the metric completion.

\end{theorem}

\begin{proof}

Define
\begin{equation}
t_n
=
\max
\left\{
\max_i x_i^{(n)},
\max_i y_i^{(n)},
|z^{(n)}-z_0|,
|w^{(n)}-z_0|
\right\}.
\label{eq:producttn}
\end{equation}

Since both sequences converge to \(p\),
\[
t_n\longrightarrow0.
\]

Define the common reference point
\[
a_n
=
(t_n,\ldots,t_n,z_0).
\]

Consider the interpolation from \(q_n\) to \(a_n\),
\[
\gamma_n(s)
=
\left(
x^{(n)}
+
s(t_n\mathbf1-x^{(n)}),
\,
z^{(n)}
+
s(z_0-z^{(n)})
\right).
\]

For every active coordinate,
\[
\gamma_{n,i}(s)
\geq
st_n.
\]

Moreover, by the definition of \(t_n\),
\[
|\dot\gamma_n(s)|
\leq
Ct_n.
\]

Therefore
\[
L_{g_\Pi}(\gamma_n)
\leq
Ct_n^{1-A_I/2}
\int_0^1
s^{-A_I/2}\,ds,
\]
and hence
\begin{equation}
d_{g_\Pi}(q_n,a_n)
\leq
Ct_n^{1-A_I/2}.
\label{eq:productqnan}
\end{equation}

The same argument gives
\[
d_{g_\Pi}(q_n',a_n)
\leq
Ct_n^{1-A_I/2}.
\]

Thus
\[
d_{g_\Pi}(q_n,q_n')
\leq
Ct_n^{1-A_I/2}.
\]

Since
\[
1-\frac{A_I}{2}>0,
\]
the right-hand side tends to zero.

\end{proof}

%================================================
\subsection{Sum Interaction}
\label{subsec:sumnormalcollapse}
%================================================

\begin{lemma}[Shrinking Normal Diameter for the Sum Metric]
\label{lem:sumshrinkingnormaldiameter}

Assume
\[
M_I<2.
\]

Then there exist constants
\[
C>0,
\qquad
t_0>0,
\]
such that
\begin{equation}
\operatorname{diam}_{g_\Sigma}
Q_t(z_0)
\leq
Ct^{1-M_I/2}
\label{eq:sumnormaldiameter}
\end{equation}
for every
\[
0<t<t_0.
\]

Consequently,
\[
\operatorname{diam}_{g_\Sigma}
Q_t(z_0)
\longrightarrow0
\qquad
\text{as }t\downarrow0.
\]

\end{lemma}

\begin{proof}

Let
\[
q=(x,z_0)\in Q_t(z_0),
\]
and let
\[
q_t
=
(t,\ldots,t,z_0).
\]

Use the interpolation path
\[
\gamma(s)
=
\bigl(
x+s(t\mathbf1-x),
z_0
\bigr).
\]

As before,
\[
\gamma_i(s)\geq st.
\]

For sufficiently small \(t\),
\[
0<st\leq1.
\]

Since
\[
\alpha_i\leq M_I,
\]
we have
\[
\gamma_i(s)^{-\alpha_i}
\leq
(st)^{-M_I}.
\]

Therefore
\[
\sum_{i\in I}
\gamma_i(s)^{-\alpha_i}
\leq
k(st)^{-M_I}.
\]

The bounded inactive contribution is absorbed into the constant.

Since
\[
|\dot\gamma(s)|
\leq
Ct,
\]
the local upper comparison gives
\[
|\dot\gamma(s)|_{g_\Sigma}
\leq
Ct(st)^{-M_I/2}.
\]

Hence
\[
L_{g_\Sigma}(\gamma)
\leq
Ct^{1-M_I/2}
\int_0^1
s^{-M_I/2}\,ds.
\]

Because
\[
M_I<2,
\]
the integral is finite. Thus
\[
d_{g_\Sigma}(q,q_t)
\leq
Ct^{1-M_I/2}.
\]

The triangle inequality through \(q_t\) gives
\[
\operatorname{diam}_{g_\Sigma}
Q_t(z_0)
\leq
Ct^{1-M_I/2}.
\]

\end{proof}

\begin{theorem}[Normal Collapse for the Sum Interaction]
\label{thm:sumnormalcollapse}

Let
\[
p=(0,z_0)\in F_I^\circ
\]
and assume
\[
M_I<2.
\]

If two interior sequences
\[
q_n\longrightarrow p,
\qquad
q_n'\longrightarrow p
\]
converge to \(p\) in the background topology, then
\begin{equation}
d_{g_\Sigma}(q_n,q_n')
\longrightarrow0.
\label{eq:sumnormalcollapsesequences}
\end{equation}

Consequently, all interior sequences converging to the same accessible
background point determine the same point of the metric completion.

\end{theorem}

\begin{proof}

Write
\[
q_n=(x^{(n)},z^{(n)}),
\qquad
q_n'=(y^{(n)},w^{(n)}),
\]
and define
\[
t_n
=
\max
\left\{
\max_i x_i^{(n)},
\max_i y_i^{(n)},
|z^{(n)}-z_0|,
|w^{(n)}-z_0|
\right\}.
\]

Then
\[
t_n\longrightarrow0.
\]

Let
\[
a_n
=
(t_n,\ldots,t_n,z_0).
\]

Using the same interpolation argument as in
Lemma~\ref{lem:sumshrinkingnormaldiameter}, now allowing tangential
variation of size at most \(t_n\), we obtain
\[
d_{g_\Sigma}(q_n,a_n)
\leq
Ct_n^{1-M_I/2}
\]
and
\[
d_{g_\Sigma}(q_n',a_n)
\leq
Ct_n^{1-M_I/2}.
\]

Hence
\[
d_{g_\Sigma}(q_n,q_n')
\leq
Ct_n^{1-M_I/2}.
\]

Since
\[
1-\frac{M_I}{2}>0,
\]
the right-hand side tends to zero.

\end{proof}

%================================================
\subsection{Shrinking Adapted Neighbourhoods}
\label{subsec:shrinkingadaptedneighbourhoods}
%================================================

For the global completion theorem, it is useful to strengthen the
normal-box estimates by allowing tangential motion.

Let
\[
p=(0,z_0)\in F_I^\circ.
\]

Choose \(t_0>0\) so small that \(\overline{U_{2t_0}(p)}\) is
contained in the fixed adapted coordinate neighbourhood \(U\) on
which the local metric comparisons hold. For \(0<t<t_0\), define
\begin{equation}
U_t(p)
=
\left\{
(x,z)\in U\cap X^\circ:
0<x_i\leq t
\text{ for every }i\in I,
\quad
|z-z_0|\leq t
\right\}.
\label{eq:shrinkingadaptedneighbourhood}
\end{equation}

\begin{proposition}[Shrinking Diameter of Accessible Interior Neighbourhoods]
\label{prop:shrinkingneighborhoods}

If
\[
A_I<2,
\]
then
\begin{equation}
\operatorname{diam}_{g_\Pi}
U_t(p)
\leq
Ct^{1-A_I/2}.
\label{eq:productshrinkingneighbourhood}
\end{equation}

If
\[
M_I<2,
\]
then
\begin{equation}
\operatorname{diam}_{g_\Sigma}
U_t(p)
\leq
Ct^{1-M_I/2}.
\label{eq:sumshrinkingneighbourhood}
\end{equation}

In both cases, the diameter tends to zero as
\[
t\downarrow0.
\]

\end{proposition}

\begin{proof}

Choose the reference point
\[
q_t
=
(t,\ldots,t,z_0).
\]

For any
\[
q=(x,z)\in U_t(p),
\]
consider
\begin{equation}
\gamma(s)
=
\left(
x+s(t\mathbf1-x),
\,
z+s(z_0-z)
\right).
\label{eq:shrinkingneighbourhoodpath}
\end{equation}

For every active coordinate,
\[
\gamma_i(s)\geq st,
\]
while
\[
|\dot\gamma(s)|
\leq
Ct.
\]

Therefore, for the product interaction,
\[
L_{g_\Pi}(\gamma)
\leq
Ct^{1-A_I/2}
\int_0^1
s^{-A_I/2}\,ds,
\]
and for the sum interaction,
\[
L_{g_\Sigma}(\gamma)
\leq
Ct^{1-M_I/2}
\int_0^1
s^{-M_I/2}\,ds.
\]

The integrals converge under the corresponding accessibility
conditions.

Thus every point of \(U_t(p)\) lies within the stated distance of
\(q_t\). The triangle inequality then gives the diameter estimates.

\end{proof}

%================================================
\subsection{Cross-Stratum Shrinking Neighbourhoods}
\label{subsec:crossstratumshrinking}
%================================================

The preceding proposition is stated for interior points. The global
completion argument will also require control of nearby accessible
points lying on lower-codimension open faces whose closures meet
\(F_I^\circ\).

We now record that extension explicitly.

\begin{lemma}[Cross-Stratum Shrinking Estimate]
\label{lem:crossstratumshrinking}

Let
\[
p\in F_I^\circ
\]
be accessible for the interaction under consideration.

Then there exist constants
\[
C>0,
\qquad
t_0>0,
\]
such that for every
\[
0<t<t_0,
\]
the following holds.

Let \(q\) and \(q'\) be accessible points of \(X\) whose background
coordinates lie in the closure of the adapted region
\[
\overline{U_t(p)}.
\]

Then the corresponding completion points satisfy
\begin{equation}
d_{\overline X}(q,q')
\leq
Ct^\delta,
\label{eq:crossstratumdiameter}
\end{equation}
where
\[
\delta
=
1-\frac{A_I}{2}
\]
for the product interaction and
\[
\delta
=
1-\frac{M_I}{2}
\]
for the sum interaction.

Consequently,
\[
\operatorname{diam}_{\overline X}
\left(
X^{\mathrm{acc}}
\cap
\overline{U_t(p)}
\right)
\longrightarrow0
\qquad
\text{as }t\downarrow0.
\]

\end{lemma}

\begin{proof}

The identification of an accessible background point with the unique
completion class represented by any interior sequence converging to it
is supplied by Proposition~\ref{prop:existenceuniquenessaccessible}
below. That proposition depends only on
Proposition~\ref{prop:shrinkingneighborhoods}, so this forward
reference is not circular.

We give the product case; the sum case is identical with \(A_I\)
replaced by \(M_I\).

Fix
\[
q,q'
\in
X^{\mathrm{acc}}
\cap
\overline{U_t(p)}.
\]
Write their adapted coordinates as
\[
q=(x,z),
\qquad
q'=(x',z'),
\]
where
\[
0\leq x_i,x_i'\leq t
\qquad
(i\in I),
\]
and
\[
|z-z_0|,
\ |z'-z_0|
\leq t.
\]
After shrinking the adapted chart around \(p\), all hypersurfaces not
indexed by \(I\) remain disjoint from this coordinate neighbourhood.
Thus making every active normal coordinate strictly positive produces
an interior point of \(X\).

Choose a sequence
\[
\varepsilon_n\downarrow0
\]
with
\[
0<\varepsilon_n<t.
\]
Define
\[
q_n
=
(x^{(n)},z),
\qquad
x_i^{(n)}
=
x_i+\varepsilon_n,
\]
and similarly
\[
q_n'
=
(x'^{(n)},z'),
\qquad
x_i'^{(n)}
=
x_i'+\varepsilon_n.
\]
Then
\[
q_n,q_n'\in X^\circ,
\qquad
q_n\longrightarrow q,
\qquad
q_n'\longrightarrow q'
\]
in the background topology. Moreover,
\[
0<x_i^{(n)},x_i'^{(n)}\leq2t,
\]
and
\[
|z-z_0|,
\ |z'-z_0|
\leq t\leq2t.
\]
Hence
\[
q_n,q_n'\in U_{2t}(p)
\]
for every \(n\).

By Proposition~\ref{prop:shrinkingneighborhoods},
\[
d_{g_\Pi}(q_n,q_n')
\leq
C(2t)^{1-A_I/2}.
\]
By Proposition~\ref{prop:existenceuniquenessaccessible}, because
\(q\) and \(q'\) are accessible, the sequences \(\{q_n\}\) and
\(\{q_n'\}\) represent the unique completion points corresponding
respectively to \(q\) and \(q'\). Passing to the metric completion
therefore gives
\[
d_{\overline X}(q,q')
\leq
C(2t)^{1-A_I/2}.
\]
Absorbing the fixed factor
\[
2^{1-A_I/2}
\]
into the constant \(C\), we obtain
\[
d_{\overline X}(q,q')
\leq
Ct^{1-A_I/2}.
\]

For the sum interaction, the same construction and
Proposition~\ref{prop:shrinkingneighborhoods} give
\[
d_{\overline X}(q,q')
\leq
C(2t)^{1-M_I/2}
\leq
Ct^{1-M_I/2}.
\]

The construction is independent of the open strata containing \(q\)
and \(q'\): coordinates corresponding to hypersurfaces that are not
active at \(q\) or \(q'\) may already be positive, while vanishing
active coordinates are perturbed by \(\varepsilon_n\). Thus the same
estimate holds uniformly across all accessible strata meeting the
adapted neighbourhood.

Taking the supremum over
\[
q,q'
\in
X^{\mathrm{acc}}
\cap
\overline{U_t(p)}
\]
proves the diameter statement.

\end{proof}

%================================================
\section{Global Metric Completion}
\label{sec:globalcompletion}
%================================================

We now pass from the local accessibility and collapse results to a
global description of the metric completion.

The key observation is that compactness of \(X\) gives global positive
lower bounds for both singular conformal metrics relative to the smooth
background metric \(g_0\). Consequently, every Cauchy sequence for
either singular metric is also Cauchy with respect to \(d_{g_0}\), and
therefore converges to a unique point of \(X\); we use only the
standard characterization of metric completion via Cauchy sequences,
as in \cite{BBI2001}.

The accessibility criteria of Section~\ref{sec:accessibility} exclude
critical and supercritical background limits, while the collapse and
shrinking-neighbourhood results of
Section~\ref{sec:normalcollapse}
show that all admissible approaches to the same accessible point
represent the same completion point.

The cross-stratum shrinking estimate will also allow us to identify the
topology of the metric completion with the subspace topology of the
accessible subset of \(X\).

Throughout this section,
\[
g\in\{g_\Pi,g_\Sigma\}
\]
denotes one of the two interaction metrics.

%================================================
\subsection{Global Comparison with the Background Metric}
\label{subsec:globalbackgroundcomparison}
%================================================

Recall that
\[
g_\Pi
=
\left(
\prod_{i=1}^{N}
\rho_i^{-\alpha_i}
\right)
g_0
\]
and
\[
g_\Sigma
=
\left(
\sum_{i=1}^{N}
\rho_i^{-\alpha_i}
\right)
g_0.
\]

Since \(X\) is compact, every defining function \(\rho_i\) is bounded
above. Choose
\[
R_i>0
\]
such that
\[
0\leq\rho_i\leq R_i
\qquad
\text{on }X.
\]

\begin{lemma}[Global Lower Metric Bounds]
\label{lem:globallowerbound}

There exist constants
\[
c_\Pi>0,
\qquad
c_\Sigma>0,
\]
such that
\begin{equation}
g_\Pi
\geq
c_\Pi g_0
\label{eq:globalproductlower}
\end{equation}
and
\begin{equation}
g_\Sigma
\geq
c_\Sigma g_0
\label{eq:globalsumlower}
\end{equation}
on \(X^\circ\).

Consequently,
\begin{equation}
d_{g_\Pi}(p,q)
\geq
\sqrt{c_\Pi}\,
d_{g_0}(p,q)
\label{eq:globalproductdistancelower}
\end{equation}
and
\begin{equation}
d_{g_\Sigma}(p,q)
\geq
\sqrt{c_\Sigma}\,
d_{g_0}(p,q)
\label{eq:globalsumdistancelower}
\end{equation}
for all
\[
p,q\in X^\circ.
\]

\end{lemma}

\begin{proof}

Since
\[
\rho_i\leq R_i
\]
and
\[
\alpha_i\geq0,
\]
we have
\[
\rho_i^{-\alpha_i}
\geq
R_i^{-\alpha_i}.
\]

Therefore
\[
\prod_{i=1}^{N}
\rho_i^{-\alpha_i}
\geq
\prod_{i=1}^{N}
R_i^{-\alpha_i}.
\]

Set
\[
c_\Pi
=
\prod_{i=1}^{N}
R_i^{-\alpha_i}.
\]

Then
\[
c_\Pi>0
\]
and
\[
g_\Pi
\geq
c_\Pi g_0.
\]

Similarly,
\[
\sum_{i=1}^{N}
\rho_i^{-\alpha_i}
\geq
\sum_{i=1}^{N}
R_i^{-\alpha_i}.
\]

Set
\[
c_\Sigma
=
\sum_{i=1}^{N}
R_i^{-\alpha_i}.
\]

Then
\[
c_\Sigma>0
\]
and
\[
g_\Sigma
\geq
c_\Sigma g_0.
\]

The distance inequalities follow by applying the corresponding
quadratic-form inequalities to the length of every absolutely
continuous curve and then taking the infimum over all connecting
curves.

\end{proof}

\begin{corollary}[Background Convergence of Singular-Metric Cauchy Sequences]
\label{cor:backgroundconvergence}

Every \(d_{g_\Pi}\)-Cauchy sequence and every
\(d_{g_\Sigma}\)-Cauchy sequence is Cauchy with respect to
\(d_{g_0}\).

Since \(X\) is compact, every such sequence converges to a unique point
of \(X\) in the background topology.

\end{corollary}

\begin{proof}

The distance inequalities
\eqref{eq:globalproductdistancelower}
and
\eqref{eq:globalsumdistancelower}
imply that
\[
d_g\text{-Cauchy}
\quad\Longrightarrow\quad
d_{g_0}\text{-Cauchy}.
\]

Because \(g_0\) is a smooth Riemannian metric on the compact manifold
with corners \(X\), the metric space
\[
(X,d_{g_0})
\]
is compact and therefore complete.

Uniqueness of the background limit follows from the metric property of
\(d_{g_0}\).

\end{proof}

%================================================
\subsection{The Canonical Background-Limit Map}
\label{subsec:canonicalbackgroundmap}
%================================================

Let
\[
\overline{(X^\circ,d_g)}
\]
denote the metric completion.

If
\[
\xi\in\overline{(X^\circ,d_g)}
\]
is represented by a \(d_g\)-Cauchy sequence
\[
\{q_n\}\subset X^\circ,
\]
Corollary~\ref{cor:backgroundconvergence} gives a unique point
\[
p\in X
\]
such that
\[
q_n\longrightarrow p
\]
in the background topology.

We therefore define
\begin{equation}
\Phi_g:
\overline{(X^\circ,d_g)}
\longrightarrow
X
\label{eq:completionbackgroundmap}
\end{equation}
by
\[
\Phi_g(\xi)=p.
\]

\begin{lemma}[Well-Definedness of the Background-Limit Map]
\label{lem:backgroundmapwelldefined}

The map
\[
\Phi_g
\]
is well defined.

\end{lemma}

\begin{proof}

Suppose
\[
\{q_n\}
\]
and
\[
\{q_n'\}
\]
represent the same completion point. Then
\[
d_g(q_n,q_n')
\longrightarrow0.
\]

By the global lower comparison,
\[
d_{g_0}(q_n,q_n')
\leq
C d_g(q_n,q_n')
\longrightarrow0.
\]

If
\[
q_n\to p
\]
and
\[
q_n'\to p'
\]
in the background topology, then
\[
d_{g_0}(p,p')
=
0.
\]

Hence
\[
p=p'.
\]

Thus the background limit depends only on the completion class.

\end{proof}

%================================================
\subsection{Inaccessible Background Limits Are Excluded}
\label{subsec:inaccessiblelimitsexcluded}
%================================================

The accessibility criteria of Section~\ref{sec:accessibility} imply
that the image of the canonical map consists only of accessible points.

\begin{proposition}[Exclusion of Inaccessible Product Limits]
\label{prop:excludeproductlimits}

Let
\[
\{q_n\}\subset X^\circ
\]
be a \(d_{g_\Pi}\)-Cauchy sequence satisfying
\[
q_n\longrightarrow p
\]
in the background topology.

Then
\[
A(p)<2.
\]

\end{proposition}

\begin{proof}

By Definition~\ref{def:accessiblepoint}, the existence of a
\(d_{g_\Pi}\)-Cauchy sequence converging to \(p\) means precisely that
\(p\) is accessible for \(g_\Pi\).

Theorem~\ref{thm:productaccessibility} therefore gives
\[
A(p)<2.
\]

\end{proof}

\begin{proposition}[Exclusion of Inaccessible Sum Limits]
\label{prop:excludesumlimits}

Let
\[
\{q_n\}\subset X^\circ
\]
be a \(d_{g_\Sigma}\)-Cauchy sequence satisfying
\[
q_n\longrightarrow p
\]
in the background topology.

Then
\[
M(p)<2.
\]

\end{proposition}

\begin{proof}

The sequence makes \(p\) accessible in the sense of
Definition~\ref{def:accessiblepoint}. Hence
Theorem~\ref{thm:sumaccessibility} gives
\[
M(p)<2.
\]

\end{proof}

Thus
\[
\Phi_{g_\Pi}
\left(
\overline{(X^\circ,d_{g_\Pi})}
\right)
\subseteq
X_{g_\Pi}^{\mathrm{acc}}
\]
and
\[
\Phi_{g_\Sigma}
\left(
\overline{(X^\circ,d_{g_\Sigma})}
\right)
\subseteq
X_{g_\Sigma}^{\mathrm{acc}}.
\]

%================================================
\subsection{Accessible Points Produce Unique Completion Points}
\label{subsec:accessiblepointscompletion}
%================================================

We now prove the converse.

\begin{proposition}[Existence and Uniqueness over Accessible Points]
\label{prop:existenceuniquenessaccessible}

Let
\[
p\in X_g^{\mathrm{acc}}.
\]

Then every interior sequence
\[
q_n\in X^\circ,
\qquad
q_n\longrightarrow p
\]
in the background topology is \(d_g\)-Cauchy.

Moreover, any two such sequences determine the same point of the metric
completion.

\end{proposition}

\begin{proof}

If
\[
p\in X^\circ,
\]
all defining functions are positive at \(p\); after shrinking to a
neighbourhood of \(p\), they are uniformly bounded away from zero.
Hence \(g\) and \(g_0\) are locally bi-Lipschitz equivalent there,
and the conclusion is immediate.

Suppose now that
\[
p\in\partial_g^{\mathrm{acc}}X.
\]

Then
\[
p\in F_I^\circ
\]
for some active index set \(I\).

For the product interaction,
\[
A_I<2,
\]
while for the sum interaction,
\[
M_I<2.
\]

Let
\[
U_t(p)
\]
be the shrinking adapted neighbourhoods from
Proposition~\ref{prop:shrinkingneighborhoods}.

Since
\[
q_n\longrightarrow p
\]
in the background topology, for every sufficiently small \(t>0\) there
exists \(N(t)\) such that
\[
q_n\in U_t(p)
\]
whenever
\[
n\geq N(t).
\]

For the product interaction,
\[
\operatorname{diam}_{g_\Pi}
U_t(p)
\leq
Ct^{1-A_I/2},
\]
and for the sum interaction,
\[
\operatorname{diam}_{g_\Sigma}
U_t(p)
\leq
Ct^{1-M_I/2}.
\]

In either case,
\[
\operatorname{diam}_{g}
U_t(p)
\longrightarrow0
\]
as
\[
t\downarrow0.
\]

Therefore \(\{q_n\}\) is \(d_g\)-Cauchy.

Now let
\[
\{q_n\}
\]
and
\[
\{q_n'\}
\]
be two interior sequences converging to the same accessible point
\(p\).

For every sufficiently small \(t>0\), both sequences eventually lie in
\(U_t(p)\). Hence
\[
d_g(q_n,q_n')
\leq
\operatorname{diam}_g U_t(p)
\]
for all sufficiently large \(n\).

Letting
\[
t\downarrow0
\]
shows that
\[
d_g(q_n,q_n')
\longrightarrow0.
\]

Thus the two sequences represent the same completion point.

\end{proof}

Consequently, every
\[
p\in X_g^{\mathrm{acc}}
\]
determines a unique completion point, which we denote by
\[
\Psi_g(p).
\]

This defines a map
\begin{equation}
\Psi_g:
X_g^{\mathrm{acc}}
\longrightarrow
\overline{(X^\circ,d_g)}.
\label{eq:inversecompletionmap}
\end{equation}

By construction,
\[
\Phi_g\circ\Psi_g
=
\operatorname{id}_{X_g^{\mathrm{acc}}}.
\]

%================================================
\subsection{Distinct Accessible Points Remain Distinct}
\label{subsec:distinctaccessiblepoints}
%================================================

\begin{proposition}[Separation of Distinct Accessible Points]
\label{prop:separationaccessiblepoints}

Let
\[
p,q\in X_g^{\mathrm{acc}}
\]
with
\[
p\neq q.
\]

Then
\[
\Psi_g(p)
\neq
\Psi_g(q).
\]

\end{proposition}

\begin{proof}

Choose interior approximating sequences
\[
p_n\longrightarrow p,
\qquad
q_n\longrightarrow q.
\]

By the global lower metric comparison,
\[
d_g(p_n,q_n)
\geq
c\,d_{g_0}(p_n,q_n)
\]
for some
\[
c>0.
\]

Passing to the lower limit gives
\[
\liminf_{n\to\infty}
d_g(p_n,q_n)
\geq
c\,d_{g_0}(p,q).
\]

Since
\[
p\neq q,
\]
we have
\[
d_{g_0}(p,q)>0.
\]

Therefore the two completion classes have strictly positive distance
and are distinct.

\end{proof}

It follows that
\[
\Psi_g
\]
is injective.

Since every completion point has an accessible background limit,
\[
\Phi_g
\]
and
\[
\Psi_g
\]
are inverse bijections.

%================================================
\subsection{Continuity of the Canonical Identification}
\label{subsec:completiontopology}
%================================================

We now compare the completion topology with the topology inherited from
\(X\).

\begin{proposition}[Continuity of the Background-Limit Map]
\label{prop:phicontinuous}

The map
\[
\Phi_g:
\overline{(X^\circ,d_g)}
\longrightarrow
X_g^{\mathrm{acc}}
\]
is continuous.

\end{proposition}

\begin{proof}

Let
\[
\xi_n\longrightarrow\xi
\]
in the metric completion.

Set
\[
p_n=\Phi_g(\xi_n),
\qquad
p=\Phi_g(\xi).
\]

For each \(n\), choose an interior point \(x_n\in X^\circ\)
representing \(\xi_n\) sufficiently far along a Cauchy representative
so that
\[
d_g(x_n,\xi_n)<\frac1n
\qquad\text{and}\qquad
d_{g_0}(x_n,p_n)<\frac1n.
\]
Likewise choose \(y_n\in X^\circ\) from a Cauchy representative of
\(\xi\) so that
\[
d_g(y_n,\xi)<\frac1n
\qquad\text{and}\qquad
d_{g_0}(y_n,p)<\frac1n.
\]
Since \(\xi_n\to\xi\) in the completion metric,
\[
d_g(x_n,y_n)
\leq
d_g(x_n,\xi_n)+d_{\overline X}(\xi_n,\xi)+d_g(\xi,y_n)
\longrightarrow0.
\]
The global lower distance comparison on interior points gives
\[
d_{g_0}(x_n,y_n)\leq C\,d_g(x_n,y_n)\longrightarrow0.
\]
Hence, by the triangle inequality,
\[
d_{g_0}(p_n,p)
\leq
d_{g_0}(p_n,x_n)+d_{g_0}(x_n,y_n)+d_{g_0}(y_n,p)
\longrightarrow0.
\]

Thus
\[
p_n\longrightarrow p
\]
in the background topology.

\end{proof}

The inverse direction requires the cross-stratum estimate proved in
Section~\ref{sec:normalcollapse}.

\begin{proposition}[Continuity of the Inverse Completion Map]
\label{prop:psicontinuous}

The map
\[
\Psi_g:
X_g^{\mathrm{acc}}
\longrightarrow
\overline{(X^\circ,d_g)}
\]
is continuous when \(X_g^{\mathrm{acc}}\) is equipped with the subspace
topology inherited from \(X\).

\end{proposition}

\begin{proof}

Let
\[
p_n\longrightarrow p
\]
in the background topology, where
\[
p_n,p\in X_g^{\mathrm{acc}}.
\]

If
\[
p\in X^\circ,
\]
then \(g\) is smooth near \(p\), and local equivalence of \(g\) and
\(g_0\) immediately gives
\[
d_{\overline X}
\bigl(
\Psi_g(p_n),
\Psi_g(p)
\bigr)
\longrightarrow0.
\]

Suppose now that
\[
p\in F_I^\circ
\subset\partial_g^{\mathrm{acc}}X.
\]

Let
\[
\delta
=
1-\frac{A_I}{2}
\]
for the product interaction, or
\[
\delta
=
1-\frac{M_I}{2}
\]
for the sum interaction.

In either case,
\[
\delta>0.
\]

Fix
\[
\varepsilon>0.
\]

By
Lemma~\ref{lem:crossstratumshrinking},
there exists \(t>0\) sufficiently small that
\[
C t^\delta
<
\varepsilon.
\]

Since
\[
p_n\longrightarrow p
\]
in the background topology, for sufficiently large \(n\),
\[
p_n
\]
lies in the closure of the adapted shrinking neighbourhood
\[
U_t(p).
\]

Both
\[
p_n
\]
and
\[
p
\]
are accessible, and for all sufficiently large \(n\) both lie in
\[
X_g^{\mathrm{acc}}
\cap
\overline{U_t(p)}.
\]
Lemma~\ref{lem:crossstratumshrinking}, applied to these two accessible
background points and to their canonical completion points under
\(\Psi_g\), therefore gives
\[
d_{\overline X}
\bigl(
\Psi_g(p_n),
\Psi_g(p)
\bigr)
\leq
Ct^\delta
<
\varepsilon.
\]

Hence
\[
\Psi_g(p_n)
\longrightarrow
\Psi_g(p)
\]
in the completion metric.

Thus \(\Psi_g\) is continuous.

\end{proof}

%================================================
\subsection{Global Completion Classification}
\label{subsec:globalcompletionclassification}
%================================================

We can now state the global completion theorem.

\begin{theorem}[Global Metric Completion]
\label{thm:globalcompletion}

Let \(X\) be a compact smooth manifold with corners with embedded
boundary hypersurfaces, and let \(g_0\) be a smooth Riemannian metric on
\(X\).

For the product interaction metric
\[
g_\Pi
=
\left(
\prod_{i=1}^{N}
\rho_i^{-\alpha_i}
\right)
g_0,
\]
the metric completion
\[
\overline{(X^\circ,d_{g_\Pi})}
\]
is canonically homeomorphic to
\begin{equation}
X_{g_\Pi}^{\mathrm{acc}}
=
X^\circ
\cup
\left\{
p\in\partial X:
A(p)<2
\right\}.
\label{eq:productcompletionaccessiblelocus}
\end{equation}

For the sum interaction metric
\[
g_\Sigma
=
\left(
\sum_{i=1}^{N}
\rho_i^{-\alpha_i}
\right)
g_0,
\]
the metric completion
\[
\overline{(X^\circ,d_{g_\Sigma})}
\]
is canonically homeomorphic to
\begin{equation}
X_{g_\Sigma}^{\mathrm{acc}}
=
X^\circ
\cup
\left\{
p\in\partial X:
M(p)<2
\right\}.
\label{eq:sumcompletionaccessiblelocus}
\end{equation}

In each case, the homeomorphism sends a completion class to the unique
background limit of any representing Cauchy sequence.

\end{theorem}

\begin{proof}

The map
\[
\Phi_g
\]
is well defined by
Lemma~\ref{lem:backgroundmapwelldefined}.

Propositions~\ref{prop:excludeproductlimits} and
\ref{prop:excludesumlimits} show that its image lies in the
corresponding accessible subset.

Proposition~\ref{prop:existenceuniquenessaccessible} shows that every
accessible point determines a completion point.

Proposition~\ref{prop:separationaccessiblepoints} shows that distinct
accessible points determine distinct completion points.

Hence
\[
\Phi_g
\]
is a bijection with inverse
\[
\Psi_g.
\]

Proposition~\ref{prop:phicontinuous} shows that
\[
\Phi_g
\]
is continuous, while
Proposition~\ref{prop:psicontinuous} shows that
\[
\Psi_g
\]
is continuous.

Therefore
\[
\Phi_g
\]
is a homeomorphism.

The explicit descriptions of the accessible subsets follow from
Theorems~\ref{thm:productaccessibility} and
\ref{thm:sumaccessibility}.

\end{proof}

%================================================
\subsection{Consequences for Boundary Incidence}
\label{subsec:completionboundaryincidence}
%================================================

The global completion theorem shows that the two interaction laws can
produce different incidence behaviour at higher-codimension open
faces.

\begin{proposition}[Accessibility of Incident Open Faces]
\label{prop:accessibilityboundaryincidence}

Let
\[
I\subseteq\{1,\ldots,N\}
\]
and suppose
\[
F_I^\circ\neq\varnothing.
\]

For the product interaction,
\[
F_I^\circ
\text{ is accessible}
\quad\Longleftrightarrow\quad
\sum_{i\in I}\alpha_i<2.
\]

Consequently, it may happen that every codimension-one hypersurface
indexed by \(I\) is individually accessible while the deeper open face
\(F_I^\circ\) is inaccessible.

For the sum interaction,
\[
F_I^\circ
\text{ is accessible}
\quad\Longleftrightarrow\quad
\max_{i\in I}\alpha_i<2.
\]

In particular, if
\[
\alpha_i<2
\qquad
\text{for every }i\in I,
\]
then
\[
F_I^\circ
\]
is accessible.

\end{proposition}

\begin{proof}

The statements follow directly from
Theorems~\ref{thm:productaccessibility} and
\ref{thm:sumaccessibility}.

For the product interaction, individual accessibility gives only
\[
\alpha_i<2
\qquad
\text{for each }i\in I.
\]

This does not imply
\[
\sum_{i\in I}\alpha_i<2.
\]

Hence accessibility may be lost when several individually subcritical
weights become simultaneously active.

For the sum interaction, if
\[
\alpha_i<2
\qquad
\text{for every }i\in I,
\]
then
\[
\max_{i\in I}\alpha_i<2.
\]

Therefore the open face
\[
F_I^\circ
\]
is accessible.

\end{proof}

\begin{remark}[Open Faces versus Set-Theoretic Intersections]
\label{rem:openfacesversusintersections}

The preceding proposition concerns the open face
\[
F_I^\circ,
\]
whose points have complete active index set exactly \(I\).

It does not imply that every point of the full set-theoretic
intersection
\[
\bigcap_{i\in I}H_i
\]
is accessible for the sum interaction.

A point
\[
p\in\bigcap_{i\in I}H_i
\]
may also lie on additional boundary hypersurfaces, so that
\[
I(p)\supsetneq I.
\]

Accessibility is then determined by the complete active set:
\[
p\text{ is accessible for }g_\Sigma
\quad\Longleftrightarrow\quad
\max_{j\in I(p)}\alpha_j<2.
\]

Thus the precise stability statement for the sum interaction is that
activating additional subcritical hypersurfaces preserves
accessibility.

\end{remark}

%================================================
\section{Boundary Snowflake Geometry and Hausdorff Dimension}
\label{sec:snowflake}
%================================================

The preceding sections determine which boundary points occur in the
metric completion, show that the active normal directions collapse over
each accessible background point, and identify the global metric
completion with the corresponding accessible subset of \(X\).

We now study the intrinsic metric geometry induced along an accessible
open face.

The main phenomenon is a snowflake deformation of the smooth tangential
metric. In the codimension-one setting, Romney formulates the
corresponding exponent through a H\"older condition for conformal
Grushin metrics; see \cite[Definition~3.3]{Romney2016}. In the
present multi-weight setting, the exponent is determined by the
effective singular weight associated with the interaction law.

For the product interaction, the relevant quantity is
\[
A_I
=
\sum_{i\in I}\alpha_i,
\]
whereas for the sum interaction it is
\[
M_I
=
\max_{i\in I}\alpha_i.
\]

A diagonal path provides the expected upper bounds, but by itself does
not exclude potentially shorter anisotropic excursions in the normal
variables. We therefore begin with model estimates that permit arbitrary
absolutely continuous competitors.

For the product model, exact dilation homogeneity determines the
boundary exponent once the distance between distinct tangential
completion points is shown to be positive and finite. For the sum
model, exact homogeneity is generally lost when the exponents are
unequal. Nevertheless, at sufficiently small scales, the largest active
exponent controls the local boundary geometry.

%================================================
\subsection{Boundary Points and Induced Completion Metrics}
\label{subsec:boundarycompletionmetric}
%================================================

Let
\[
F_I^\circ
\]
be an accessible open face, and let
\[
d_{F_I}
\]
denote the Riemannian distance induced by the restriction of \(g_0\) to
\(F_I^\circ\).

Fix
\[
p\in F_I^\circ
\]
and choose adapted local coordinates
\[
(x,z)
=
\bigl((x_i)_{i\in I},z\bigr),
\]
where the variables \(x_i\) are normal to the active boundary
hypersurfaces and \(z\) denotes coordinates along the face.

For a tangential point \(z\) and \(\varepsilon>0\), define
\begin{equation}
p_\varepsilon(z)
=
(\varepsilon,\ldots,\varepsilon,z).
\label{eq:boundaryapproximation}
\end{equation}

By
Theorems~\ref{thm:productnormalcollapse} and
\ref{thm:sumnormalcollapse},
the family
\[
\{p_\varepsilon(z)\}_{\varepsilon\downarrow0}
\]
determines a unique point of the corresponding metric completion
whenever the face is accessible.

We denote this completion point by
\[
\overline z.
\]

For two points \(z,z'\) of the same accessible open face, define the
induced boundary distance by
\begin{equation}
d_\partial(z,z')
=
d_{\overline X}(\overline z,\overline{z'}),
\label{eq:inducedboundarydistance}
\end{equation}
where \(\overline X\) denotes the relevant metric completion.

Equivalently,
\begin{equation}
d_\partial(z,z')
=
\lim_{\varepsilon\downarrow0}
d_g
\bigl(
p_\varepsilon(z),
p_\varepsilon(z')
\bigr).
\label{eq:boundarydistancelimit}
\end{equation}

%================================================

%================================================
\subsection{The Homogeneous Multi-Normal Product Model}
\label{subsec:homogeneousproductboundary}
%================================================

Consider
\[
\mathcal H
=
(0,\infty)^k\times\mathbb R^m
\]
with coordinates
\[
(x,z)
=
(x_1,\ldots,x_k,z),
\]
equipped with the homogeneous product metric
\begin{equation}
g_{\boldsymbol{\alpha}}^\Pi
=
\left(
\prod_{i=1}^{k}x_i^{-\alpha_i}
\right)
\left(
\sum_{i=1}^{k}dx_i^2+|dz|^2
\right),
\label{eq:homogeneousmultiproductmetric}
\end{equation}
where
\[
\alpha_i\geq0.
\]

Set
\begin{equation}
A
=
\sum_{i=1}^{k}\alpha_i
\label{eq:homogeneousproductweight}
\end{equation}
and assume
\[
A<2.
\]

Define
\begin{equation}
\delta_\Pi
=
1-\frac{A}{2}.
\label{eq:homogeneousproductdelta}
\end{equation}

Then
\[
0<\delta_\Pi\leq1.
\]

The formal boundary
\[
\{x_1=\cdots=x_k=0\}
\]
lies at finite distance, and every
\[
z\in\mathbb R^m
\]
determines a unique completion point.

%================================================
\subsubsection{Exact Dilation Homogeneity}
%================================================

For \(s>0\), define
\begin{equation}
D_s(x,z)
=
(sx,sz).
\label{eq:homogeneousproductdilation}
\end{equation}

Then
\begin{align}
D_s^\ast g_{\boldsymbol{\alpha}}^\Pi
&=
\left(
\prod_{i=1}^{k}(sx_i)^{-\alpha_i}
\right)
\left(
\sum_{i=1}^{k}s^2dx_i^2+s^2|dz|^2
\right)
\\
&=
s^{2-A}
g_{\boldsymbol{\alpha}}^\Pi.
\label{eq:homogeneousproductmetricdilation}
\end{align}

Therefore
\begin{equation}
d_{g_{\boldsymbol{\alpha}}^\Pi}
(D_sp,D_sq)
=
s^{\delta_\Pi}
d_{g_{\boldsymbol{\alpha}}^\Pi}(p,q).
\label{eq:homogeneousproductdistancedilation}
\end{equation}

Equation~\eqref{eq:homogeneousproductdistancedilation} shows that
\(D_s\) is globally Lipschitz with respect to the model distance, with
Lipschitz inverse \(D_{1/s}\). Hence it maps Cauchy sequences to
Cauchy sequences and preserves equivalence of Cauchy sequences. It
therefore extends uniquely to a bijective similarity of the metric
completion, with the same scaling factor \(s^{\delta_\Pi}\).

%================================================
\subsubsection{Nondegeneracy of the Product Boundary Metric}
%================================================

\begin{lemma}[Positive and Finite Product Boundary Distance]
\label{lem:positiveproductunitboundarydistance}

Let
\[
z,z'\in\mathbb R^m
\]
with
\[
L=|z-z'|>0.
\]

Then there exist constants
\[
0<c\leq C<\infty
\]
depending only on the model parameters such that
\begin{equation}
cL^{\delta_\Pi}
\leq
d_{\partial,\Pi}
(\overline z,\overline{z'})
\leq
CL^{\delta_\Pi}.
\label{eq:productboundarytwosidedmodel}
\end{equation}

In particular, if
\[
e\in\mathbb R^m
\]
is a unit vector, then
\begin{equation}
0
<
d_{\partial,\Pi}
(\overline0,\overline e)
<
\infty.
\label{eq:positiveproductunitdistance}
\end{equation}

\end{lemma}

\begin{proof}

We first prove the upper bound.

For \(\varepsilon>0\), consider
\[
p_\varepsilon(z)
=
(\varepsilon,\ldots,\varepsilon,z)
\]
and
\[
p_\varepsilon(z')
=
(\varepsilon,\ldots,\varepsilon,z').
\]

Fix
\[
R>\varepsilon.
\]

Join \(p_\varepsilon(z)\) to
\[
(R,\ldots,R,z)
\]
along the normal diagonal, move tangentially to
\[
(R,\ldots,R,z'),
\]
and then return along the normal diagonal to
\(p_\varepsilon(z')\).

Along either normal segment,
\[
x_1=\cdots=x_k=s,
\]
so
\[
\prod_{i=1}^{k}x_i^{-\alpha_i}
=
s^{-A}
\]
and
\[
\sum_{i=1}^{k}dx_i^2
=
k\,ds^2.
\]

Hence each normal segment has length
\[
\sqrt{k}
\int_\varepsilon^R
s^{-A/2}\,ds
=
\frac{\sqrt{k}}{\delta_\Pi}
\left(
R^{\delta_\Pi}
-
\varepsilon^{\delta_\Pi}
\right).
\]

Along the tangential segment,
\[
x_1=\cdots=x_k=R,
\]
so the conformal factor is
\[
R^{-A}
\]
and the line-element multiplier is
\[
R^{-A/2}.
\]

Thus the tangential segment has length
\[
LR^{-A/2}.
\]

Therefore
\begin{equation}
d_{g_{\boldsymbol{\alpha}}^\Pi}
\bigl(
p_\varepsilon(z),
p_\varepsilon(z')
\bigr)
\leq
\frac{2\sqrt{k}}{\delta_\Pi}
\left(
R^{\delta_\Pi}
-
\varepsilon^{\delta_\Pi}
\right)
+
LR^{-A/2}.
\label{eq:productdetourfiniteepsilon}
\end{equation}

Letting
\[
\varepsilon\downarrow0
\]
gives
\begin{equation}
d_{\partial,\Pi}
(\overline z,\overline{z'})
\leq
\frac{2\sqrt{k}}{\delta_\Pi}
R^{\delta_\Pi}
+
LR^{-A/2}.
\label{eq:productdetourupper}
\end{equation}

Choosing
\[
R=L
\]
yields
\[
d_{\partial,\Pi}
(\overline z,\overline{z'})
\leq
\left(
\frac{2\sqrt{k}}{\delta_\Pi}+1
\right)
L^{\delta_\Pi}.
\]

We now prove the lower bound.

Let
\[
\gamma_\varepsilon(t)
=
(x(t),z(t)),
\qquad
t\in[0,1],
\]
be an arbitrary absolutely continuous curve joining
\(p_\varepsilon(z)\) to \(p_\varepsilon(z')\).

Define
\begin{equation}
r(t)
=
\left(
\sum_{i=1}^{k}x_i(t)^2
\right)^{1/2}
\label{eq:productmodelradius}
\end{equation}
and
\[
R
=
\sup_{t\in[0,1]}r(t).
\]

Since
\[
x_i(t)\leq r(t)
\]
for every \(i\),
\[
\prod_{i=1}^{k}
x_i(t)^{-\alpha_i/2}
\geq
r(t)^{-A/2}.
\]

Moreover,
\[
|\dot r(t)|
\leq
|\dot x(t)|
\]
for almost every \(t\).

Therefore
\begin{equation}
L_{g_{\boldsymbol{\alpha}}^\Pi}
(\gamma_\varepsilon)
\geq
\int_0^1
r(t)^{-A/2}
\sqrt{
|\dot r(t)|^2
+
|\dot z(t)|^2
}
\,dt.
\label{eq:productcombinedintegrandlower}
\end{equation}

Since
\[
r(t)\leq R,
\]
we obtain
\begin{align}
L_{g_{\boldsymbol{\alpha}}^\Pi}
(\gamma_\varepsilon)
&\geq
R^{-A/2}
\int_0^1
|\dot z(t)|\,dt
\\
&\geq
LR^{-A/2}.
\label{eq:producttangentiallower}
\end{align}

Choose
\[
t_0\in[0,1]
\]
such that
\[
r(t_0)=R.
\]

At the initial point,
\[
r(0)
=
\sqrt{k}\,\varepsilon
=:
r_\varepsilon.
\]

Restricting
\eqref{eq:productcombinedintegrandlower}
to \([0,t_0]\) and dropping the tangential component gives
\begin{equation}
L_{g_{\boldsymbol{\alpha}}^\Pi}
(\gamma_\varepsilon)
\geq
\int_0^{t_0}
r(t)^{-A/2}
|\dot r(t)|\,dt.
\label{eq:productradialsubinterval}
\end{equation}

Applying Lemma~\ref{lem:weightedradialvariation} with
\[
\beta=\frac{A}{2}
\]
gives
\[
\int_0^{t_0}
r(t)^{-A/2}
|\dot r(t)|\,dt
\geq
\frac{
R^{\delta_\Pi}
-
r_\varepsilon^{\delta_\Pi}
}{
\delta_\Pi
},
\qquad
\delta_\Pi=1-\frac{A}{2}.
\]

Thus
\begin{equation}
L_{g_{\boldsymbol{\alpha}}^\Pi}
(\gamma_\varepsilon)
\geq
\frac{
R^{\delta_\Pi}
-
r_\varepsilon^{\delta_\Pi}
}{
\delta_\Pi
}.
\label{eq:productradiallower}
\end{equation}

Combining
\eqref{eq:producttangentiallower}
and
\eqref{eq:productradiallower},
\begin{equation}
L_{g_{\boldsymbol{\alpha}}^\Pi}
(\gamma_\varepsilon)
\geq
\max
\left\{
LR^{-A/2},
\frac{
R^{\delta_\Pi}
-
r_\varepsilon^{\delta_\Pi}
}{
\delta_\Pi
}
\right\}.
\label{eq:productcombinedlower}
\end{equation}

If
\[
R\leq L,
\]
then
\[
LR^{-A/2}
\geq
L^{1-A/2}
=
L^{\delta_\Pi}.
\]

If
\[
R>L,
\]
then
\[
R^{\delta_\Pi}
\geq
L^{\delta_\Pi},
\]
and therefore
\[
L_{g_{\boldsymbol{\alpha}}^\Pi}
(\gamma_\varepsilon)
\geq
\frac{
L^{\delta_\Pi}
-
r_\varepsilon^{\delta_\Pi}
}{
\delta_\Pi
}.
\]

Thus there exist constants
\[
c_1,c_2>0
\]
depending only on \(A\) such that
\begin{equation}
L_{g_{\boldsymbol{\alpha}}^\Pi}
(\gamma_\varepsilon)
\geq
c_1L^{\delta_\Pi}
-
c_2r_\varepsilon^{\delta_\Pi}.
\label{eq:productuniformlower}
\end{equation}

Taking the infimum over all admissible curves and then letting
\[
\varepsilon\downarrow0
\]
gives
\[
d_{\partial,\Pi}
(\overline z,\overline{z'})
\geq
cL^{\delta_\Pi}.
\]

Together with the upper bound, this proves
\eqref{eq:productboundarytwosidedmodel}.

\end{proof}

%================================================
\subsubsection{Exact Product Snowflake Law}
%================================================

\begin{theorem}[Exact Homogeneous Product Snowflake Law]
\label{thm:exactproductsnowflake}

Suppose
\[
A<2.
\]

Then there exists a constant
\[
C_{\boldsymbol{\alpha},k}>0
\]
such that for all
\[
z,z'\in\mathbb R^m,
\]
\begin{equation}
d_{\partial,\Pi}
(\overline z,\overline{z'})
=
C_{\boldsymbol{\alpha},k}
|z-z'|^{1-A/2}.
\label{eq:exactproductsnowflake}
\end{equation}

\end{theorem}

\begin{remark}[Relation to Conformal Grushin Spaces]
\label{rem:romneyspecialcase}

When \(k=1\), the homogeneous product model
\(g_{\boldsymbol{\alpha}}^\Pi\) is a conformal Grushin-type metric in
the sense of \cite{Romney2016}, obtained from the Euclidean metric by
the conformal factor \(d_E(\cdot,Y)^{-\beta}\) with \(Y\) the singular
hyperplane and \(\beta=\alpha/2\). Under this identification, the
exponent \(1-A/2\) in
Theorem~\ref{thm:exactproductsnowflake} coincides with the exponent
\(1-\beta\) appearing in the H\"older condition of
\cite[Definition~3.3]{Romney2016}. The relationship is not one of
mere restatement. Romney's general theory assumes the H\"older
condition of \cite[Definition~3.3]{Romney2016} (and gives sufficient
conditions for it in \cite[Proposition~3.4]{Romney2016}) for a general
closed singular set \(Y\subset\mathbb R^n\). By contrast,
Theorem~\ref{thm:exactproductsnowflake} proves an \emph{exact}
proportionality in the flat-hyperplane model, available because the
model enjoys full dilation and translation symmetry. The exact
law is thus a stronger statement special to this symmetric case, not a
generalization of the Hölder estimate to arbitrary codimension.

\end{remark}

\begin{proof}

Translation invariance in the tangential variables implies that the
boundary distance depends only on
\[
z-z'.
\]

Rotational invariance in the Euclidean \(z\)-variables implies that it
depends only on
\[
L=|z-z'|.
\]

Let
\[
e\in\mathbb R^m
\]
be a unit vector and define
\[
C_{\boldsymbol{\alpha},k}
=
d_{\partial,\Pi}
(\overline0,\overline e).
\]

By
Lemma~\ref{lem:positiveproductunitboundarydistance},
\[
0
<
C_{\boldsymbol{\alpha},k}
<
\infty.
\]

For \(L>0\), the dilation \(D_L\) maps the boundary pair
\[
(\overline0,\overline e)
\]
to
\[
(\overline0,\overline{Le}).
\]

Using
\eqref{eq:homogeneousproductdistancedilation},
\[
d_{\partial,\Pi}
(\overline0,\overline{Le})
=
L^{\delta_\Pi}
d_{\partial,\Pi}
(\overline0,\overline e).
\]

Therefore
\[
d_{\partial,\Pi}
(\overline z,\overline{z'})
=
C_{\boldsymbol{\alpha},k}
|z-z'|^{\delta_\Pi}.
\]

\end{proof}

\begin{remark}[Anisotropic Normal Excursions]
\label{rem:productanisotropicexcursions}

The proof does not assume that minimizing or nearly minimizing curves
satisfy
\[
x_1=\cdots=x_k.
\]

The normal coordinates may move to different depths. Such anisotropic
motion may affect the multiplicative constant
\(C_{\boldsymbol{\alpha},k}\), but exact dilation homogeneity forces the
power-law exponent
\[
1-\frac12
\sum_{i=1}^{k}\alpha_i.
\]

\end{remark}

%================================================
\subsection{Local Product Boundary Geometry}
\label{subsec:localproductboundary}
%================================================

We now return to the product interaction on \(X^\circ\).

Let
\[
F_I^\circ
\]
be an accessible open face, so that
\[
A_I<2.
\]

\begin{theorem}[Local Product Boundary Snowflake Geometry]
\label{thm:localproductsnowflake}

For every
\[
p\in F_I^\circ
\]
there exists a relatively compact neighbourhood
\[
V\Subset F_I^\circ
\]
of \(p\) and constants
\[
0<c\leq C<\infty
\]
such that for all sufficiently close
\[
z,z'\in V,
\]
\begin{equation}
c\,
d_{F_I}(z,z')^{1-A_I/2}
\leq
d_{\partial,\Pi}(z,z')
\leq
C\,
d_{F_I}(z,z')^{1-A_I/2}.
\label{eq:localproductsnowflake}
\end{equation}

\end{theorem}

\begin{proof}

Choose adapted coordinates near \(p\).

By the local metric comparison of
Section~\ref{subsec:localmodelcomparisons},
there exist constants
\[
0<c_0\leq C_0<\infty
\]
such that
\[
c_0
g_{\boldsymbol{\alpha}_I}^\Pi
\leq
g_\Pi
\leq
C_0
g_{\boldsymbol{\alpha}_I}^\Pi
\]
on a sufficiently small coordinate neighbourhood, where
\[
\boldsymbol{\alpha}_I
=
(\alpha_i)_{i\in I}.
\]

Thus the local singular metric is uniformly comparable to the
homogeneous product model.

By
Theorem~\ref{thm:exactproductsnowflake},
the model boundary metric is exactly proportional to
\[
|z-z'|^{1-A_I/2}.
\]

The smooth face distance satisfies
\[
|z-z'|
\asymp
d_{F_I}(z,z')
\]
locally.

Hence
\[
d_{\partial,\Pi}(z,z')
\asymp
d_{F_I}(z,z')^{1-A_I/2}
\]
for sufficiently close tangential points.

Finally, the global lower comparison
\[
g_\Pi\geq c g_0
\]
implies that curves leaving a fixed adapted coordinate neighbourhood
incur a positive length cost. To make this quantitative, fix relatively compact adapted coordinate
neighborhoods \(V\Subset U\) containing the boundary points under
consideration. Since \(g_\Pi\geq c\,g_0\) and
\[
\operatorname{dist}_{g_0}(\overline V,X\setminus U)>0,
\]
there exists \(E_U>0\) such that every curve joining two points of
\(V\) and leaving \(U\) has \(g_\Pi\)-length at least \(E_U\).
The local construction gives
\[
d_{\partial,\Pi}(z,z')
\leq C\,d_{F_I}(z,z')^{\delta_\Pi},
\qquad
\delta_\Pi=1-\frac{A_I}{2}.
\]
Choose \(\eta>0\) such that
\[
C\eta^{\delta_\Pi}<E_U.
\]
Hence, whenever
\[
d_{F_I}(z,z')<\eta,
\]
every curve leaving \(U\) has length at least
\(E_U>C\,d_{F_I}(z,z')^{\delta_\Pi}\), whereas a local competitor
has length at most \(C\,d_{F_I}(z,z')^{\delta_\Pi}\). Thus an
escaping curve cannot lower the infimum. Moreover, escaping competitors
already satisfy
\[
L_g(\gamma)\geq E_U
>
C\,d_{F_I}(z,z')^{\delta_\Pi}
\geq
c\,d_{F_I}(z,z')^{\delta_\Pi},
\]
so the lower bound is also valid globally. Therefore the local
two-sided estimate governs the global completion distance whenever
\(d_{F_I}(z,z')<\eta\).

This proves the theorem.

\end{proof}

%================================================
\subsection{The Multi-Normal Sum Model}
\label{subsec:multisummodel}
%================================================

Consider
\begin{equation}
g_{\boldsymbol{\alpha}}^\Sigma
=
\left(
\sum_{i=1}^{k}x_i^{-\alpha_i}
\right)
\left(
\sum_{i=1}^{k}dx_i^2+|dz|^2
\right)
\label{eq:multisummodel}
\end{equation}
on
\[
(0,\infty)^k\times\mathbb R^m,
\]
where
\[
\alpha_i\geq0.
\]

Set
\begin{equation}
M
=
\max_{1\leq i\leq k}\alpha_i
\label{eq:summodelmax}
\end{equation}
and assume
\[
M<2.
\]

Define
\begin{equation}
\delta_\Sigma
=
1-\frac{M}{2}.
\label{eq:summodeldelta}
\end{equation}

Then
\[
0<\delta_\Sigma\leq1.
\]

\begin{remark}[A Related Dominance Phenomenon in Spectral Theory]
\label{rem:cdvdtdominance}

The principle that a single dominant exponent governs the leading-order
behaviour near a singular boundary, even when the degeneracy is not
spatially homogeneous, also appears in a different context: for
metrics \(g=u^{-\alpha}\bar g\) with a single boundary hypersurface but
a continuously \emph{varying} exponent \(\alpha\) along the boundary,
\cite{CdVDT2026} show that the Weyl asymptotics of the associated
Laplace--Beltrami operator are governed by the points where \(\alpha\)
is maximal. The present setting is different in kind \textemdash{} here
the exponent is piecewise constant across finitely many hypersurfaces
meeting at a corner, and the question is metric completion rather than
spectral asymptotics \textemdash{} but the qualitative phenomenon, that
extremal rather than averaged behaviour controls the leading-order
answer, is the same.

\end{remark}

When the exponents are unequal, the metric
\eqref{eq:multisummodel}
is generally not exactly homogeneous under isotropic dilation. Indeed,
\[
D_s^\ast g_{\boldsymbol{\alpha}}^\Sigma
=
\left(
\sum_{i=1}^{k}
s^{2-\alpha_i}x_i^{-\alpha_i}
\right)
\left(
\sum_{i=1}^{k}dx_i^2+|dz|^2
\right),
\]
so in general no single scaling exponent factors from the full metric.

Nevertheless, the maximal exponent \(M\) controls the local singular
behaviour near the boundary.

\begin{remark}[A Related Dominance Phenomenon in Spectral Theory]
\label{rem:cdvdtdominance-second}

The principle that a single dominant exponent governs the leading-order
behaviour near a singular boundary, even when the degeneracy is not
spatially homogeneous, also appears in a different context: for
metrics \(g=u^{-\alpha}\bar g\) with a single boundary hypersurface but
a continuously \emph{varying} exponent \(\alpha\) along the boundary,
\cite{CdVDT2026} show that the Weyl asymptotics of the associated
Laplace--Beltrami operator are governed by the points where \(\alpha\)
is maximal. The present setting is different in kind \textemdash{} here
the exponent is piecewise constant across finitely many hypersurfaces
meeting at a corner, and the question is metric completion rather than
spectral asymptotics \textemdash{} but the qualitative phenomenon, that
extremal rather than averaged behaviour controls the leading-order
answer, is the same.

\end{remark}

%================================================
\subsubsection{Local Boundary Snowflake Law for the Sum Model}
%================================================

\begin{theorem}[Local Boundary Snowflake Law for the Sum Model]
\label{thm:multisumsnowflake}

Suppose
\[
M<2.
\]

Fix
\[
L_0>0.
\]

Then there exist constants
\[
0<c\leq C<\infty
\]
depending only on
\[
\boldsymbol{\alpha},
\quad
k,
\quad
L_0,
\]
such that for all
\[
z,z'\in\mathbb R^m
\]
satisfying
\[
0<L:=|z-z'|\leq L_0,
\]
one has
\begin{equation}
cL^{1-M/2}
\leq
d_{\partial,\Sigma}
(\overline z,\overline{z'})
\leq
CL^{1-M/2}.
\label{eq:multisumsnowflake}
\end{equation}

\end{theorem}

\begin{proof}

Set
\[
L=|z-z'|.
\]

We first prove the upper bound.

For \(\varepsilon>0\), consider
\[
p_\varepsilon(z)
=
(\varepsilon,\ldots,\varepsilon,z)
\]
and
\[
p_\varepsilon(z')
=
(\varepsilon,\ldots,\varepsilon,z').
\]

Choose the detour height
\[
R=L.
\]

Along the normal diagonal,
\[
x_1=\cdots=x_k=s.
\]

For
\[
0<s\leq L_0,
\]
we have
\[
s^{-\alpha_i}
=
s^{-M}s^{M-\alpha_i}.
\]

By definition,
\[
M=\max_{1\leq j\leq k}\alpha_j,
\]
so
\[
\alpha_i\leq M
\qquad
\text{for every }i.
\]
Hence
\[
M-\alpha_i\geq0,
\]
\[
s^{M-\alpha_i}
\leq
\max\{1,L_0^M\}.
\]

Thus
\begin{equation}
\sum_{i=1}^{k}
s^{-\alpha_i}
\leq
C_0s^{-M},
\label{eq:summodelsmallscalecomparison}
\end{equation}
where
\[
C_0
=
k\max\{1,L_0^M\}.
\]

The two normal portions therefore have total length at most
\[
C
L^{1-M/2}.
\]

At the tangential slice
\[
x_1=\cdots=x_k=L,
\]
the conformal factor satisfies
\[
\sum_{i=1}^{k}
L^{-\alpha_i}
\leq
C_0L^{-M},
\]
so the tangential segment has length at most
\[
CL^{1-M/2}.
\]

Letting
\[
\varepsilon\downarrow0
\]
gives
\begin{equation}
d_{\partial,\Sigma}
(\overline z,\overline{z'})
\leq
CL^{1-M/2}.
\label{eq:summodelupper}
\end{equation}

We now prove the lower bound.

Let
\[
\gamma_\varepsilon(t)
=
(x(t),z(t)),
\qquad
t\in[0,1],
\]
be an arbitrary absolutely continuous curve joining
\(p_\varepsilon(z)\) to \(p_\varepsilon(z')\).

Define
\[
r(t)
=
\left(
\sum_{i=1}^{k}x_i(t)^2
\right)^{1/2}
\]
and
\[
R
=
\sup_{t\in[0,1]}r(t).
\]

Choose \(j\) such that
\[
\alpha_j=M.
\]

Since
\[
x_j(t)\leq r(t)\leq R,
\]
we have
\[
\sum_{i=1}^{k}
x_i(t)^{-\alpha_i}
\geq
x_j(t)^{-M}
\geq
R^{-M}.
\]

Hence
\begin{align}
L_{g_{\boldsymbol{\alpha}}^\Sigma}
(\gamma_\varepsilon)
&\geq
R^{-M/2}
\int_0^1
|\dot z(t)|\,dt
\\
&\geq
LR^{-M/2}.
\label{eq:summodeltangentiallower}
\end{align}

Also,
\[
x_j(t)^{-M}
\geq
r(t)^{-M},
\]
and
\[
|\dot r(t)|
\leq
|\dot x(t)|.
\]

Therefore
\[
L_{g_{\boldsymbol{\alpha}}^\Sigma}
(\gamma_\varepsilon)
\geq
\int_0^1
r(t)^{-M/2}
|\dot r(t)|\,dt.
\]

Choose
\[
t_0\in[0,1]
\]
such that
\[
r(t_0)=R.
\]

Since
\[
r(0)
=
\sqrt{k}\,\varepsilon
=:
r_\varepsilon,
\]
we obtain
\[
L_{g_{\boldsymbol{\alpha}}^\Sigma}
(\gamma_\varepsilon)
\geq
\int_0^{t_0}
r(t)^{-M/2}
|\dot r(t)|\,dt.
\]

Applying Lemma~\ref{lem:weightedradialvariation} with
\[
\beta=\frac{M}{2}
\]
gives
\[
L_{g_{\boldsymbol{\alpha}}^\Sigma}
(\gamma_\varepsilon)
\geq
\frac{
R^{\delta_\Sigma}
-
r_\varepsilon^{\delta_\Sigma}
}{
\delta_\Sigma
},
\qquad
\delta_\Sigma=1-\frac{M}{2}.
\]

The two estimates are complementary: the tangential bound is
strongest when the maximal normal excursion \(R\) is small relative to
the tangential separation \(L\), whereas the radial bound controls
large excursions. Combining them gives
\begin{equation}
L_{g_{\boldsymbol{\alpha}}^\Sigma}
(\gamma_\varepsilon)
\geq
\max
\left\{
LR^{-M/2},
\frac{
R^{\delta_\Sigma}
-
r_\varepsilon^{\delta_\Sigma}
}{
\delta_\Sigma
}
\right\}.
\label{eq:summodelcombinedlower}
\end{equation}

If
\[
R\leq L,
\]
then
\[
LR^{-M/2}
\geq
L^{1-M/2}.
\]

If
\[
R>L,
\]
then
\[
R^{\delta_\Sigma}
\geq
L^{\delta_\Sigma}.
\]

Thus there exist constants
\[
c_1,c_2>0
\]
such that
\[
L_{g_{\boldsymbol{\alpha}}^\Sigma}
(\gamma_\varepsilon)
\geq
c_1L^{\delta_\Sigma}
-
c_2r_\varepsilon^{\delta_\Sigma}.
\]

Taking the infimum over all admissible curves and letting
\[
\varepsilon\downarrow0
\]
gives
\[
d_{\partial,\Sigma}
(\overline z,\overline{z'})
\geq
cL^{\delta_\Sigma}.
\]

Together with
\eqref{eq:summodelupper},
this proves
\eqref{eq:multisumsnowflake}.

\end{proof}

\begin{remark}[Local Nature of the Sum Estimate]
\label{rem:summodellocality}

The bounded-scale restriction in
Theorem~\ref{thm:multisumsnowflake}
is essential when the exponents are unequal.

Near the singular boundary, the largest exponent
\[
M=\max_i\alpha_i
\]
controls the relevant singular order. At large normal scales, however,
terms with smaller exponents decay more slowly and may dominate the
sum. Consequently, \(M\) need not determine the large-scale geometry
of the unbounded model.

The theorem therefore asserts a local boundary snowflake law, which is
precisely the form required for the analysis of accessible open faces
in the compact manifold-with-corners setting.

\end{remark}

%================================================
\subsection{Local Sum Boundary Geometry}
\label{subsec:localsumboundary}
%================================================

Let
\[
F_I^\circ
\]
be accessible for the sum interaction, so that
\[
M_I<2.
\]

\begin{theorem}[Local Sum Boundary Snowflake Geometry]
\label{thm:localsumsnowflake}

For every
\[
p\in F_I^\circ
\]
there exists a relatively compact neighbourhood
\[
V\Subset F_I^\circ
\]
of \(p\) and constants
\[
0<c\leq C<\infty
\]
such that for all sufficiently close
\[
z,z'\in V,
\]
\begin{equation}
c\,
d_{F_I}(z,z')^{1-M_I/2}
\leq
d_{\partial,\Sigma}(z,z')
\leq
C\,
d_{F_I}(z,z')^{1-M_I/2}.
\label{eq:localsumsnowflake}
\end{equation}

\end{theorem}

\begin{proof}

Choose adapted coordinates near \(p\).

The inactive defining functions remain bounded above and below by
positive constants on a sufficiently small relatively compact
coordinate neighbourhood, while \(g_0\) is uniformly comparable to the
Euclidean coordinate metric.

Thus the local singular metric is uniformly comparable, on sufficiently
small scales, to the multi-normal sum model with active weights
\[
(\alpha_i)_{i\in I}.
\]

By
Theorem~\ref{thm:multisumsnowflake},
the model boundary metric satisfies
\[
d_{\partial,\Sigma}^{\mathrm{model}}(z,z')
\asymp
|z-z'|^{1-M_I/2}
\]
for sufficiently small tangential separations.

Since
\[
|z-z'|
\asymp
d_{F_I}(z,z')
\]
locally on the smooth open face, we obtain
\[
d_{\partial,\Sigma}(z,z')
\asymp
d_{F_I}(z,z')^{1-M_I/2}.
\]

As in the product case, the global lower comparison
\[
g_\Sigma\geq c g_0
\]
implies a positive length cost for curves leaving a fixed local
coordinate neighbourhood. To make the locality explicit, fix relatively compact coordinate
neighborhoods \(V\Subset U\) containing the boundary points under
consideration. Since \(g\geq c\,g_0\) and the \(g_0\)-distance from
\(\overline V\) to \(X\setminus U\) is positive, there is a constant
\(E_U>0\) such that every curve joining two points of \(V\) and leaving
\(U\) has \(g\)-length at least \(E_U\). If the local upper bound is
\(C\,d_{F_I}(z,z')^\delta\), choose \(\eta>0\) so that
\[
C\eta^\delta<E_U.
\]
Then, whenever \(d_{F_I}(z,z')<\eta\), every curve leaving \(U\)
has length at least
\[
E_U>C\,d_{F_I}(z,z')^\delta,
\]
whereas a local competitor has length at most
\(C\,d_{F_I}(z,z')^\delta\). Hence an escaping curve cannot lower the
infimum. Moreover, since the local two-sided constants may be chosen
with \(0<c\leq C\), every escaping competitor already satisfies
\[
L_g(\gamma)\geq E_U
>
C\,d_{F_I}(z,z')^\delta
\geq
c\,d_{F_I}(z,z')^\delta.
\]
Thus both the lower and upper local estimates remain valid for the
global completion distance. The phrase ``sufficiently close'' may
therefore be taken to mean \(d_{F_I}(z,z')<\eta\).

\end{proof}

%================================================
\subsection{Persistence of Tangential Geometry}
\label{subsec:persistencetangentialgeometry}
%================================================

The preceding results show that collapse of the active normal
directions does not identify distinct tangential points.

\begin{corollary}[Persistence of Accessible Open Faces]
\label{cor:persistencefaces}

Let
\[
F_I^\circ
\]
be an accessible open face.

If
\[
z\neq z'
\]
belong to \(F_I^\circ\), then the corresponding completion points are
distinct.

\end{corollary}

\begin{proof}

For sufficiently close points, this follows directly from the lower
bounds in
Theorems~\ref{thm:localproductsnowflake} and
\ref{thm:localsumsnowflake}.

For arbitrary distinct points of the same open face, separation follows
from the global lower metric comparison of
Lemma~\ref{lem:globallowerbound}.

\end{proof}

%================================================
\subsection{Hausdorff Dimension of Accessible Open Faces}
\label{subsec:boundaryhausdorffdimension}
%================================================

The local snowflake laws determine the Hausdorff dimension of accessible
open faces equipped with their induced completion metrics.

We use the following standard snowflake-dimension fact~\cite{heinonen2001}. If
\[
d'
\asymp
d^\delta
\]
locally, with
\[
0<\delta\leq1,
\]
then a set of \(d\)-diameter \(r\) has \(d'\)-diameter comparable to
\(r^\delta\). Consequently, the \(s\)-dimensional Hausdorff sums for
\(d'\) are comparable, at small scales, to the
\(\delta s\)-dimensional Hausdorff sums for \(d\). Hence
\[
\dim_H(X,d')
=
\frac{1}{\delta}\dim_H(X,d).
\]
For a smooth \(m\)-dimensional Riemannian manifold,
\[
\dim_H(X,d)=m,
\]
and therefore
\[
\dim_H(X,d')
=
\frac{m}{\delta}.
\]

Since every open face is second countable, it admits a countable cover
by relatively compact coordinate neighbourhoods on each of which the
same local snowflake exponent holds. Using
\[
\dim_H\!\left(\bigcup_{j=1}^{\infty}E_j\right)
=
\sup_j\dim_H(E_j),
\]
and the fact that each nonempty coordinate neighbourhood has the same
local Hausdorff dimension, the Hausdorff dimension of the entire open
face is that common value.

\begin{corollary}[Product Boundary Hausdorff Dimension]
\label{cor:productboundarydimension}

Let
\[
F_I^\circ
\]
be an accessible open face of dimension
\[
m
\]
for the product interaction.

Then
\begin{equation}
\dim_H
\bigl(
F_I^\circ,
d_{\partial,\Pi}
\bigr)
=
\frac{m}{
1-\frac{A_I}{2}
}.
\label{eq:productboundarydimension}
\end{equation}

\end{corollary}

\begin{proof}

By
Theorem~\ref{thm:localproductsnowflake},
the induced metric is locally bi-Lipschitz equivalent to
\[
d_{F_I}^{\,1-A_I/2}.
\]

Hence
\[
\dim_H
\bigl(
F_I^\circ,
d_{\partial,\Pi}
\bigr)
=
\frac{m}{
1-A_I/2
}.
\]

\end{proof}

\begin{corollary}[Sum Boundary Hausdorff Dimension]
\label{cor:sumboundarydimension}

Let
\[
F_I^\circ
\]
be an accessible open face of dimension
\[
m
\]
for the sum interaction.

Then
\begin{equation}
\dim_H
\bigl(
F_I^\circ,
d_{\partial,\Sigma}
\bigr)
=
\frac{m}{
1-\frac{M_I}{2}
}.
\label{eq:sumboundarydimension}
\end{equation}

\end{corollary}

\begin{proof}

By
Theorem~\ref{thm:localsumsnowflake},
the induced metric is locally bi-Lipschitz equivalent to
\[
d_{F_I}^{\,1-M_I/2}.
\]

Hence
\[
\dim_H
\bigl(
F_I^\circ,
d_{\partial,\Sigma}
\bigr)
=
\frac{m}{
1-M_I/2
}.
\]

\end{proof}

\begin{remark}[Facewise Nature of the Dimension Formula]
\label{rem:facewisehausdorffdimension}

The preceding formulas concern an individual accessible open face
equipped with the induced completion metric.

No single formula is asserted for the Hausdorff dimension of the entire
completion boundary. Different accessible strata may have different
smooth dimensions and different effective singular weights, and hence
different local Hausdorff dimensions.

\end{remark}

%================================================
\section{Boundary Incidence and Deeper Corner Geometry}
\label{sec:incidence}
%================================================

The preceding sections determine which open faces occur at finite
distance and describe the intrinsic metric geometry carried by those
that survive in the completion. We now study how the two interaction
laws affect the incidence relations among boundary strata.

The distinction is most pronounced when several individually
subcritical hypersurfaces meet.

For the product interaction, the active weights accumulate:
\[
A_I
=
\sum_{i\in I}\alpha_i.
\]
Thus a collection of individually accessible hypersurfaces may determine
an inaccessible deeper open face if the cumulative weight crosses the
critical threshold \(2\).

For the sum interaction, the effective weight is
\[
M_I
=
\max_{i\in I}\alpha_i.
\]
Hence activating additional hypersurfaces with subcritical weights does
not destroy accessibility.

Throughout this section, incidence statements are formulated in terms
of open faces and complete active index sets. A full set-theoretic
intersection of boundary hypersurfaces may contain deeper strata lying
on additional hypersurfaces, and accessibility at such points is always
determined by the complete active index set.

%================================================
\subsection{Loss of Incidence for the Product Interaction}
\label{subsec:productincidenceloss}
%================================================

Let
\[
I\subseteq\{1,\ldots,N\}
\]
and suppose
\[
F_I^\circ\neq\varnothing.
\]

Recall that
\[
F_I^\circ
\text{ is accessible for }g_\Pi
\quad\Longleftrightarrow\quad
A_I<2.
\]

Thus accessibility need not be preserved when additional positive
weights become active.

\begin{proposition}[Loss of Accessibility at a Deeper Product Face]
\label{prop:productincidenceloss}

Suppose
\[
I=\{i_1,\ldots,i_k\}
\]
and
\[
F_I^\circ\neq\varnothing.
\]

Assume
\[
\alpha_{i_j}<2
\qquad
\text{for every }j,
\]
but
\[
A_I
=
\sum_{j=1}^{k}\alpha_{i_j}
\geq2.
\]

Then each codimension-one hypersurface
\[
H_{i_j}^\circ
\]
is individually accessible for the product interaction, while the
deeper open face
\[
F_I^\circ
\]
is inaccessible.

\end{proposition}

\begin{proof}

For each \(j\),
\[
\alpha_{i_j}<2,
\]
so
\[
H_{i_j}^\circ
\]
is accessible by
Theorem~\ref{thm:productaccessibility}.

On the other hand,
\[
A_I\geq2,
\]
so the open face
\[
F_I^\circ
\]
is inaccessible by the same theorem.

\end{proof}

\begin{example}[Codimension-Two Incidence Loss]
\label{ex:codimensiontwoincidenceloss}

Suppose two boundary hypersurfaces \(H_1\) and \(H_2\) meet and
\[
\alpha_1<2,
\qquad
\alpha_2<2,
\]
but
\[
\alpha_1+\alpha_2\geq2.
\]

Then the open hypersurface strata
\[
H_1^\circ,
\qquad
H_2^\circ
\]
are both accessible for the product metric, whereas their common
codimension-two open face
\[
F_{\{1,2\}}^\circ
\]
is inaccessible.

Thus the finite-distance completion retains the two incident
hypersurface strata but omits their deeper common open face.

\end{example}

%================================================
\subsection{Minimal Inaccessible Product Faces}
\label{subsec:minimalinaccessiblefaces}
%================================================

The preceding phenomenon may occur first at any codimension.

\begin{definition}[Minimal Inaccessible Product Face]
\label{def:minimalinaccessibleproductface}

An open face
\[
F_I^\circ
\]
is called \emph{minimally inaccessible for the product interaction} if
\[
A_I\geq2
\]
while
\[
A_J<2
\]
for every proper subset
\[
J\subsetneq I
\]
for which
\[
F_J^\circ\neq\varnothing.
\]

\end{definition}

Thus all proper incident open faces are accessible, but the first
simultaneous activation of the complete set \(I\) crosses the critical
threshold.

\begin{proposition}[Existence Criterion for Minimal Inaccessibility]
\label{prop:minimalinaccessibilitycriterion}

Let
\[
F_I^\circ\neq\varnothing.
\]

Then \(F_I^\circ\) is minimally inaccessible for the product
interaction precisely when
\[
\sum_{i\in I}\alpha_i\geq2
\]
and
\[
\sum_{j\in J}\alpha_j<2
\]
for every proper incident active set
\[
J\subsetneq I.
\]

\end{proposition}

\begin{proof}

This is an immediate consequence of the product accessibility criterion
applied to \(I\) and to each proper incident active set \(J\).

\end{proof}

%================================================
\subsection{A Homogeneous Codimension-Two Product Model}
\label{subsec:homogeneousincidentfacemodel}
%================================================

To study the metric behaviour of two accessible incident faces near an
inaccessible or critical common corner, consider the homogeneous model
\[
Q
=
(0,\infty)^2
\]
with metric
\begin{equation}
g_{\alpha,\beta}
=
x^{-\alpha}y^{-\beta}
\left(
dx^2+dy^2
\right),
\label{eq:incidenthomogeneousmetric}
\end{equation}
where
\[
0\leq\alpha<2,
\qquad
0\leq\beta<2.
\]

The individual boundary rays
\[
\{x=0,y>0\},
\qquad
\{y=0,x>0\}
\]
are therefore accessible.

Set
\begin{equation}
A
=
\alpha+\beta.
\label{eq:incidenttotalweight}
\end{equation}

For \(s>0\), let
\[
P_s
\]
denote the completion point corresponding to the boundary point
\[
(0,s),
\]
and let
\[
Q_s
\]
denote the completion point corresponding to
\[
(s,0).
\]

We study the distance
\[
d(P_s,Q_s)
\]
as
\[
s\downarrow0.
\]

%================================================
\subsection{Exact Scaling of Incident-Face Distance}
\label{subsec:incidentfacescaling}
%================================================

Define
\[
D_s(x,y)
=
(sx,sy).
\]

Then
\begin{align}
D_s^\ast g_{\alpha,\beta}
&=
(sx)^{-\alpha}(sy)^{-\beta}
s^2
\left(
dx^2+dy^2
\right)
\\
&=
s^{2-(\alpha+\beta)}
g_{\alpha,\beta}
\\
&=
s^{2-A}
g_{\alpha,\beta}.
\end{align}

Therefore
\begin{equation}
d(D_sp,D_sq)
=
s^{1-A/2}
d(p,q).
\label{eq:incidentdistancescaling}
\end{equation}

The same identity shows that \(D_s\) is a similarity of the interior
metric space with Lipschitz inverse \(D_{1/s}\). It therefore extends
uniquely to a bijective similarity of the metric completion, so the
notation \(D_s(P_1)\) and \(D_s(Q_1)\) below is well defined.

Since
\[
D_s(P_1)=P_s,
\qquad
D_s(Q_1)=Q_s,
\]
we obtain
\begin{equation}
d(P_s,Q_s)
=
s^{1-A/2}
d(P_1,Q_1).
\label{eq:incidentexactscaling}
\end{equation}

To use this formula, we must verify that
\[
d(P_1,Q_1)
\]
is strictly positive and finite.

%================================================
\subsection{Nondegeneracy of the Unit Incident-Face Distance}
\label{subsec:incidentunitdistance}
%================================================

\begin{lemma}[Positive and Finite Unit Incident-Face Distance]
\label{lem:positiveincidentunitdistance}

For the homogeneous metric
\[
g_{\alpha,\beta}
=
x^{-\alpha}y^{-\beta}
(dx^2+dy^2),
\]
with
\[
0\leq\alpha<2,
\qquad
0\leq\beta<2,
\]
the completion points
\[
P_1=(0,1),
\qquad
Q_1=(1,0)
\]
satisfy
\begin{equation}
0
<
d(P_1,Q_1)
<
\infty.
\label{eq:positiveincidentunitdistance}
\end{equation}

\end{lemma}

\begin{proof}

We first prove finiteness.

Choose
\[
R>1.
\]

Starting near \(P_1\), move a short finite-distance segment into the
interior, travel through a compact interior region joining a
neighbourhood of \((0,1)\) to a neighbourhood of \((1,0)\), and then
approach \(Q_1\).

Since
\[
\alpha<2
\]
and
\[
\beta<2,
\]
both boundary rays are individually accessible. Hence the two endpoint
segments have finite length, while the middle segment lies in a compact
subset of the interior where the metric is smooth.

Thus
\[
d(P_1,Q_1)<\infty.
\]

We now prove strict positivity.

Let
\[
\gamma
\]
be any absolutely continuous interior curve joining sufficiently close
interior approximations of \(P_1\) and \(Q_1\).

Set
\[
r
=
\sqrt{x^2+y^2}.
\]

We split into two cases.

First suppose that
\[
r(t)\leq2
\]
along the entire curve.

Then
\[
x(t)\leq2,
\qquad
y(t)\leq2,
\]
so
\[
x(t)^{-\alpha}y(t)^{-\beta}
\geq
2^{-(\alpha+\beta)}.
\]

Hence
\[
g_{\alpha,\beta}
\geq
2^{-A}
(dx^2+dy^2)
\]
along the curve.

The Euclidean distance between interior points approaching
\[
(0,1)
\]
and
\[
(1,0)
\]
tends to
\[
\sqrt2.
\]

We now make the lower bound uniform in the endpoint regularization.

Choose the regularization parameter so small
that the interior approximations \(P_\varepsilon,Q_\varepsilon\)
satisfy
\[
|P_\varepsilon-(0,1)|<\frac14,
\qquad
|Q_\varepsilon-(1,0)|<\frac14.
\]
Then
\[
|P_\varepsilon-Q_\varepsilon|
\geq \sqrt2-\frac12=:d_0>0.
\]
If the curve remains in \(r\leq2\), the estimate
\(g_{\alpha,\beta}\geq2^{-A}(dx^2+dy^2)\) gives
\[
L_{g_{\alpha,\beta}}(\gamma)
\geq2^{-A/2}d_0=:c_0>0.
\]
If the curve reaches \(r=2\), decrease the same regularization
threshold if necessary so that \(r_{\mathrm{initial}}\leq5/4\). The
length accumulated before the first hitting time of \(r=2\) then
obeys, by Lemma~\ref{lem:weightedradialvariation},
\[
L_{g_{\alpha,\beta}}(\gamma)
\geq\int_{5/4}^{2}s^{-A/2}\,ds=:c_1>0.
\]
Thus every sufficiently regularized competitor satisfies the uniform
bound
\[
L_{g_{\alpha,\beta}}(\gamma)\geq\min\{c_0,c_1\}>0,
\]
independently of the endpoint regularization.

Passing to the completion limit gives
\[
d(P_1,Q_1)>0.
\]

\end{proof}

%================================================
\subsection{Incident-Face Trichotomy}
\label{subsec:incidentfacetrichotomy}
%================================================

Combining
\eqref{eq:incidentexactscaling}
with
Lemma~\ref{lem:positiveincidentunitdistance}
gives a complete trichotomy.

\begin{remark}
The preceding positivity lemma uses only the individual accessibility
assumptions \(\alpha<2\) and \(\beta<2\). It is independent of the value
of \(A=\alpha+\beta\), and therefore applies uniformly in the
subcritical, critical, and supercritical regimes distinguished below.
\end{remark}

\begin{theorem}[Incident-Face Distance Trichotomy]
\label{thm:incidentfacetrichotomy}

Let
\[
A=\alpha+\beta.
\]

Then
\begin{equation}
d(P_s,Q_s)
=
C_{\alpha,\beta}
s^{1-A/2},
\label{eq:incidentfacetrichotomyformula}
\end{equation}
where
\[
C_{\alpha,\beta}
=
d(P_1,Q_1)
\in(0,\infty).
\]

Consequently:

\begin{enumerate}
    \item If
    \[
    A<2,
    \]
    then
    \[
    d(P_s,Q_s)\longrightarrow0.
    \]

    \item If
    \[
    A=2,
    \]
    then
    \[
    d(P_s,Q_s)
    =
    C_{\alpha,\beta}>0
    \]
    for every \(s>0\).

    \item If
    \[
    A>2,
    \]
    then
    \[
    d(P_s,Q_s)\longrightarrow+\infty.
    \]
\end{enumerate}

\end{theorem}

\begin{proof}

Equation
\eqref{eq:incidentfacetrichotomyformula}
follows directly from
\eqref{eq:incidentexactscaling}
and the positivity and finiteness of
\[
d(P_1,Q_1).
\]

The three cases follow from the sign of
\[
1-\frac{A}{2}.
\]

\end{proof}

%================================================
\subsection{Higher-Codimension Product Incidence}
\label{subsec:highercodimensionproductincidence}
%================================================

The same cumulative-weight mechanism extends to higher codimension.

Let
\[
F_I^\circ
\]
be an open face with
\[
A_I\geq2.
\]

Suppose every relevant proper incident open face
\[
F_J^\circ,
\qquad
J\subsetneq I,
\]
is accessible.

Then the finite-distance completion contains those proper incident
faces but omits
\[
F_I^\circ.
\]

If
\[
A_I=2,
\]
the corresponding homogeneous scaling is critical.

If
\[
A_I>2,
\]
the deeper corner is supercritical.

A complete metric description of all mutual distances among multiple
incident higher-codimension faces may depend on the detailed active
weight configuration. Nevertheless, the accessibility threshold itself
is governed solely by
\[
A_I
=
\sum_{i\in I}\alpha_i.
\]

%================================================
\subsection{Stability under Additional Subcritical Hypersurfaces}
\label{subsec:sumsubcriticalstability}
%================================================

The sum interaction behaves differently.

\begin{proposition}[Subcritical Activation Stability for the Sum Interaction]
\label{prop:sumsubcriticalstability}

Let
\[
I\subseteq\{1,\ldots,N\}
\]
and suppose
\[
F_I^\circ\neq\varnothing.
\]

If
\[
\alpha_i<2
\qquad
\text{for every }i\in I,
\]
then
\[
F_I^\circ
\]
is accessible for the sum interaction.

More generally, suppose an accessible active set \(I\) is enlarged to
\[
I\cup J
\]
by activating additional hypersurfaces whose weights satisfy
\[
\alpha_j<2
\qquad
\text{for every }j\in J.
\]

Then every open face with complete active set
\[
I\cup J
\]
is accessible for the sum interaction.

\end{proposition}

\begin{proof}

Since
\[
\alpha_i<2
\qquad
\text{for every }i\in I,
\]
we have
\[
M_I
=
\max_{i\in I}\alpha_i
<2.
\]

Hence
\[
F_I^\circ
\]
is accessible by
Theorem~\ref{thm:sumaccessibility}.

For the enlarged active set,
\[
M_{I\cup J}
=
\max
\left\{
M_I,
\max_{j\in J}\alpha_j
\right\}.
\]

Every term in this maximum is strictly less than \(2\). Therefore
\[
M_{I\cup J}<2,
\]
and the corresponding open face is accessible.

\end{proof}

\begin{remark}[No Claim for Arbitrary Set-Theoretic Intersections]
\label{rem:noarbitraryintersectionclaim}

The proposition concerns open faces and their complete active index
sets.

It does not assert that every point of a set-theoretic intersection
\[
H_{i_1}\cap\cdots\cap H_{i_k}
\]
is accessible merely because the listed hypersurfaces are individually
accessible.

A point in this intersection may lie on additional boundary
hypersurfaces. Accessibility is always determined by the full active
index set
\[
I(p).
\]

\end{remark}

%================================================
\section{Conclusion and Outlook}
For multi-weighted conformal metrics on manifolds with corners, the
interaction law determines the effective singular order:
\[
A_I=\sum_{i\in I}\alpha_i\quad\text{for }g_\Pi,\qquad
M_I=\max_{i\in I}\alpha_i\quad\text{for }g_\Sigma.
\]
The preceding results show that these quantities govern finite-distance
accessibility, collapse of active normal directions, metric completion,
boundary incidence, and the local snowflake geometry of accessible faces.
The product law accumulates active singularities, while the sum law is
controlled by the strongest active weight.

\begin{table}[htbp]
\centering
\small
\begin{tabular}{|p{3.2cm}|p{5.0cm}|p{5.0cm}|}
\hline
 & \textbf{Product interaction} & \textbf{Sum interaction} \\
\hline
Effective active weight
&
\(A_I=\sum_{i\in I}\alpha_i\)
&
\(M_I=\max_{i\in I}\alpha_i\)
\\
\hline
Accessibility
&
\(A_I<2\)
&
\(M_I<2\)
\\
\hline
Boundary exponent
&
\(1-A_I/2\)
&
\(1-M_I/2\)
\\
\hline
Higher-codimension accessibility effect
&
Cumulative active weight
&
Dominant active weight
\\
\hline
Metric separation of incident strata
&
Codimension-two trichotomy proved in the homogeneous model
&
No analogous separation theorem asserted here
\\
\hline
\end{tabular}
\caption{Comparison of the product and sum interaction laws.}
\label{tab:interactioncomparison}
\end{table}

Several directions remain open. A primary open problem is to determine
whether the codimension-two metric-separation trichotomy extends to
higher codimension, where incident strata are indexed by different
subsets of the active hypersurfaces and the separation behaviour may
depend on both the active weights and the incidence combinatorics.
Other natural next questions include the behaviour of geodesics near
critical strata, analytic properties of the associated
Laplace--Beltrami operators (in the spirit of the Weyl-law analyses
for singular metrics in
\cite{CdVDdHT2024,DietzeRead2024,Dietze2025,CdVDT2026}, though those
works treat a single degeneracy exponent rather than the multi-weight
corner interactions studied here), curvature and measure-theoretic
effects of multi-weight interactions, and extensions to more general
nonconformal or variable-exponent singular metrics. The present
results provide a geometric framework in which such questions can be
posed while keeping accessibility, completion topology, and induced
boundary geometry on a common footing.

\section*{Acknowledgements}
The author acknowledges the University of Malaya, where he received his
undergraduate education in mathematics. The mathematical training received
there provided an important foundation for the author's continued independent
study and research.

\section*{Declarations}

\begin{itemize}
\item Funding: The author received no funding for this research.
\item Conflict of interest: The author declares no conflict of interest.
\item Ethics approval and consent to participate: Not applicable.
\item Consent for publication: Not applicable.
\item Data availability: Not applicable; this manuscript reports no data.
\item Materials availability: Not applicable.
\item Code availability: Not applicable.
\item Author contribution: M.F.Z.A. is the sole author and is responsible for all aspects of this work.
\end{itemize}

\bibliographystyle{amsplain}
\bibliography{references}

\end{document}